\documentclass[a4paper,11pt,reqno]{customart}

\usepackage[utf8x]{inputenc}
\usepackage[margin=1in]{geometry}
\usepackage{verbatim}
\usepackage{color}
\usepackage[table]{xcolor}
\usepackage{listings}
\usepackage{amssymb,amsfonts,amsthm,amsmath}
\usepackage{url}
\usepackage{enumerate}
\usepackage{todonotes}
\usepackage[all]{xy}
\usepackage{multirow}
\usepackage{attachfile}
\usepackage{stmaryrd}
\usepackage{footnote}
\makesavenoteenv{tabular}
\usepackage{hyperref}

\newtheorem{theorem}{Theorem}[section]
 
\newtheorem{lemma}[theorem]{Lemma}
\newtheorem{corollary}[theorem]{Corollary}
\newtheorem{example}[theorem]{Example}
\newtheorem*{claim}{Claim}

\theoremstyle{definition}
\newtheorem{definition}[theorem]{Definition}

\newtheorem{question}[theorem]{Question}

\theoremstyle{remark}
\newtheorem{remark}[theorem]{Remark}

\newcommand{\imp}{\rightarrow}
\newcommand{\Imp}{\Rightarrow}
\newcommand{\biimp}{\leftrightarrow}
\newcommand{\Nb}{\mathbb{N}}

\newcommand{\Psf}{\mathsf{P}}
\newcommand{\Qsf}{\mathsf{Q}}
\newcommand{\Acal}{\mathcal{A}}

\newcommand{\Ccal}{\mathcal{C}}

\newcommand{\Mcal}{\mathcal{M}}
\newcommand{\Pcal}{\mathcal{P}}
\newcommand{\Rcal}{\mathcal{R}}
\newcommand{\Scal}{\mathcal{S}}

\newcommand{\cs}{2^\Nb}
\newcommand{\str}{2^{<\Nb}}
\newcommand{\uh}{{\upharpoonright}}

\newcommand{\llb}{\llbracket}
\newcommand{\rrb}{\rrbracket}
\newcommand{\cyl}[1]{\llb {#1} \rrb}
\renewcommand{\setminus}{\smallsetminus}

\newcommand{\bad}{\mathrm{Bad}}

\newcommand{\set}[1]{\left\{ #1 \right\}}
\newcommand{\card}[1]{\left| #1 \right|}

\newcommand{\tuple}[1]{\left\langle #1 \right\rangle}

\newcommand{\cond}[1]{\left\{\begin{array}{ll} #1 \end{array}\right.}

\newcommand{\s}[1]{\ensuremath{\sf{#1}}}
\newcommand{\ran}[1]{#1\mbox{-}\s{RAN}}
\newcommand{\wwkls}[1]{#1\mbox{-}\s{WWKL}}

\DeclareMathOperator{\h}{\operatorname{-}}
\DeclareMathOperator{\rca}{\s{RCA}_0}
\DeclareMathOperator{\aca}{\s{ACA}_0}
\DeclareMathOperator{\wkl}{\s{WKL}_0}

\DeclareMathOperator{\dnr}{\s{DNR}}
\DeclareMathOperator{\dnrzp}{\s{DNR}[\emptyset']}
\DeclareMathOperator{\isig}{\s{I}\Sigma}
\DeclareMathOperator{\bsig}{\s{B}\Sigma}
\DeclareMathOperator{\bst}{\bsig^0_2}
\DeclareMathOperator{\ist}{\s{I}\Sigma^0_2}
\DeclareMathOperator{\cst}{\s{C}\Sigma^0_2}
\DeclareMathOperator{\rkl}{\s{RKL}}
\DeclareMathOperator{\rwkl}{\s{RWKL}}
\DeclareMathOperator{\rwwkl}{\s{RWWKL}}

\DeclareMathOperator{\rt}{\s{RT}}
\DeclareMathOperator{\srt}{\s{SRT}}

\DeclareMathOperator{\rrt}{\s{RRT}}
\DeclareMathOperator{\ads}{\s{ADS}}
\DeclareMathOperator{\sads}{\s{SADS}}
\DeclareMathOperator{\cads}{\s{CADS}}
\DeclareMathOperator{\cac}{\s{CAC}}

\DeclareMathOperator{\coh}{\s{COH}}
\DeclareMathOperator{\pizog}{\Pi^0_1\s{G}}

\DeclareMathOperator{\fpf}{\s{FPF}}
\DeclareMathOperator{\ts}{\s{TS}}
\DeclareMathOperator{\sts}{\s{STS}}
\DeclareMathOperator{\srrt}{\s{SRRT}}
\DeclareMathOperator{\wsrrt}{\s{WSRRT}}
\DeclareMathOperator{\fs}{\s{FS}}

\DeclareMathOperator{\sfs}{\s{SFS}}
\DeclareMathOperator{\hyp}{\s{HYP}}
\DeclareMathOperator{\emo}{\s{EM}}

\DeclareMathOperator{\semo}{\s{SEM}}

\DeclareMathOperator{\opt}{\s{OPT}}
\DeclareMathOperator{\amt}{\s{AMT}}

\DeclareMathOperator{\ipt}{\s{IPT}}

\def§#1§{\text{#1}}%

\hypersetup{
	pdfborder={0 0 0}
}

\def\str{2^{<\Nb}}
\def\N{\mathbb{N}}

\usepackage{xcolor}	
\usepackage{soul}

\makeindex

\title{Somewhere over the rainbow Ramsey theorem for pairs}
\author{
  Ludovic Patey
}
\date{\today}

\begin{document}

\begin{abstract}
The rainbow Ramsey theorem states that every coloring of tuples where each color is used
a bounded number of times has an infinite subdomain on which no color appears twice.
The restriction of the statement to colorings over pairs ($\mathsf{RRT}^2_2$) admits several characterizations:
it is equivalent to finding an infinite subset of a 2-random, to diagonalizing against Turing machines
with the halting set as oracle...
In this paper we study principles that are closely related to the rainbow Ramsey theorem, 
the Erd\H{o}s Moser theorem and the thin set theorem
within the framework of reverse mathematics. We prove that 
the thin set theorem for pairs implies $\mathsf{RRT}^2_2$,
and that the stable thin set theorem for pairs implies the atomic model theorem
over~$\mathsf{RCA}_0$.
We define different notions of stability for the rainbow Ramsey theorem and establish characterizations in terms of Ramsey-type K\"onig's lemma, relativized Schnorr randomness
or diagonalization of $\Delta^0_2$ functions.
\end{abstract}

\maketitle


\section{Introduction}

Reverse mathematics is a vast mathematical program whose goal is to
find the provability content of theorems. Empirically, many ``ordinary'' 
(i.e. non set-theoretic) theorems happen to require very weak axioms,
and furthermore to be equivalent to one of five main subsystems of second order arithmetic.
However, among theorems studied in reverse mathematics, Ramseyan principles
are known to contradict this observation. Their computational complexities
are difficult to tackle and the introduction to a new Ramseyan principle often leads
to a new subsystem of second order arithmetics.

The Ramsey theorem for pairs ($\rt^2_2$) states that for every coloring
of pairs into two  colors, there exists an infinite restriction of the domain
on which the coloring is monochromatic. This principle benefited
of a particular attention from the scientific community \cite{Bovykin2005strength,Cholak2001strength,Chong2014metamathematics,Jockusch1972Ramseys,Liu2010Cone, Seetapun1995strength}.
The questions of its relations with $\wkl$ -- K\"onig's lemma
restricted to binary trees -- and $\srt^2_2$ -- the restriction
of $\rt^2_2$ to stable colorings -- have been opened for decades
and went recently solved. Liu~\cite{Liu2012RT22} proved that
$\rt^2_2$ does not imply for~$\wkl$ over~$\rca$, and Chong et Slaman~\cite{Chong2014metamathematics} proved
that~$\srt^2_2$ does not imply $\rt^2_2$, using non-standard models.
It remains open whether every $\omega$-model of $\srt^2_2$
is also model of $\rt^2_2$.

\bigskip

\subsection{The rainbow Ramsey theorem}

Among the consequences of Ramsey's theorem, the rainbow Ramsey theorem
intuitively states the existence of an infinite injective restriction
of any function which is already close to being injective.
We now provide its formal definition.

\begin{definition}[Rainbow Ramsey theorem]
Fix $n, k \in \N$. A coloring function $f: [\N]^n \to \N$ is \emph{$k$-bounded}
  if for every $y \in \N$, $\card{f^{-1}(y)} \leq k$. A set $R$ is a \emph{rainbow} for $f$ (or an \emph{$f$-rainbow})
  if $f$ is injective over~$[R]^n$.
  $\rrt^n_k$ is the statement ``Every $k$-bounded function $f : [\N]^n \to \N$
  has an infinite $f$-rainbow''.
  $\rrt$ is the statement: $(\forall n)(\forall k)\rrt^n_k$.
\end{definition}

A proof of the rainbow Ramsey theorem is due to Galvin who noticed that
it follows easily from Ramsey's theorem.
Hence every computable 2-bounded coloring function $f$ over $n$-tuples has an infinite $\Pi^0_n$ rainbow.
Csima and Mileti proved in \cite{Csima2009strength} that every 2-random bounds
an~$\omega$-model of $\rrt^2_2$
and deduced that  $\rrt^2_2$ implies neither $\sads$ nor $\wkl$ over $\omega$-models.
Conidis \& Slaman adapted in~\cite{Conidis2013Random} the argument from Cisma and Mileti 
to obtain $\rca \vdash \ran{2} \imp \rrt^2_2$.

Wang proved in~\cite{WangSome,Wang2013Rainbow} that $\rca + \rrt^3_2 \nvdash \aca$
and $\rrt^2_2$ is $\Pi^1_1$-conservative over $\rca + \bst$.
He refined his result in \cite{Wang2014Cohesive}, proving that $\rrt^3_2$
implies neither $\wkl$ nor $\rrt^4_2$ over~$\omega$-models.

Csima and Mileti proved in \cite{Csima2009strength} that
for every $n \in \Nb$, there exists a computable 2-bounded 
coloring over~$[\Nb]^n$ with no infinite $\Sigma^0_n$ rainbow.
Conidis \& Slaman proved in \cite{Conidis2013Random} that $\rca + \rrt^2_2$ proves $\cst$.
Later, Slaman proved in~\cite{Slam2011First} that $\rrt^2_2$ -- in fact even $\ran{2}$ --
does not imply $\bst$ over $\rca$.

In a computational perspective, Miller \cite{MillerPersonal} proved
that $\rrt^2_2$ is equivalent to $\dnrzp$ where $\dnr[\emptyset^{(n)}]$
is the statement ``for every set $X$, there exists
a function $f$ such that $f(e) \neq \Phi^{X^{(n)}}_e(e)$ for every $e$''.
$\dnrzp$ is known to be equivalent to the ability to escape
finite $\Sigma^0_2$ sets of uniformly bounded size, to diagonalize against partial $\emptyset'$-computable functions,
to find an infinite subset of a 2-random, or an infinite subset of a path in a $\Delta^0_2$ tree of positive measure.

Having so many simple characterizations speak in favor of the
naturality of the rainbow Ramsey theorem for pairs.
Its characterizations are formally stated in sections~\ref{sect:srrt22-dnr} and~\ref{srrt22-konig}
and are adapted to obtain characterizations of stable versions of the rainbow Ramsey theorem.

\bigskip

\subsection{Somewhere over \texorpdfstring{$\rrt^2_2$}{RRT22}}

There exist several proofs of the rainbow Ramsey theorem, partly due to 
the variety of its characterizations. Among them are statements about
graph theory and thin set theorem.

The \emph{Erd\H{o}s-Moser theorem} ($\emo$) states that every infinite tournament (see below)
has an infinite transitive subtournament. It can be seen
as the ability to find an infinite subdomain of an arbitrary 2-coloring of pairs
on which the coloring behaves like a linear order. It is why $\emo$, together with the ascending descending sequence principle ($\ads$),
proves $\rt^2_2$ over $\rca$.

\begin{definition}[Erd\H{os}s-Moser theorem] 
A tournament $T$ on a domain $D \subseteq \N$ is an irreflexive binary relation on~$D$ such that for all $x,y \in D$ with $x \not= y$, exactly one of $T(x,y)$ or $T(y,x)$ holds. A tournament $T$ is \emph{transitive} if the corresponding relation~$T$ is transitive in the usual sense. A tournament $T$ is \emph{stable} if $(\forall x \in D)[(\forall^{\infty} s) T(x,s) \vee (\forall^{\infty} s) T(s, x)]$.
$\emo$ is the statement ``Every infinite tournament $T$ has an infinite transitive subtournament.''
$\semo$ is the restriction of $\emo$ to stable tournaments.
\end{definition}

Bovykin and Weiermann proved in \cite{Bovykin2005strength} that $\emo + \ads$
is equivalent to $\rt^2_2$ over $\rca$, and the same equivalence holds between the stable versions.
Lerman \& al. \cite{Lerman2013Separating} proved over $\rca + \bst$ that $\emo$ implies $\opt$
and that there is an $\omega$-model of $\emo$ not model of $\srt^2_2$.
Kreuzer proved in~\cite{Kreuzer2012Primitive} that $\semo$ implies
$\bst$ over $\rca$. Bienvenu et al.~\cite{Bienvenu2014Ramsey}
and Flood \& Towsner~\cite{Flood2014Separating} proved independently that $\rca \vdash \semo \imp \rwkl$,
hence there is an $\omega$-model of $\rrt^2_2$ not model of $\semo$.
We prove that that $\emo$ implies $\rrt^2_2$ over $\rca$ using both a direct proof
and the equivalence between $\rrt^2_2$ and $\dnrzp$. 
We also prove that $\rca \vdash \emo \imp [\sts(2) \vee \coh]$.

The \emph{thin set theorem} ($\ts$) states that every coloring of tuples has a restriction over an infinite domain
on which it avoids a color. It is often studied together with the free set theorem $\fs$.
Its study has been initiated by Friedman
in the FOM mailing list \cite{FriedmFom:53:free,Friedman2010Boolean}.

\begin{definition}[Free set theorem]
Let $k \in \Nb$ and $f : [\Nb]^k \to \Nb$.
A set $A$ is \textit{free for $f$} (or \emph{$f$-free}) if for every $x_1 < \dots < x_k \in A$,
if $f(x_1, \dots, x_k) \in A$ then $f(x_1, \dots, x_k) \in \set{x_1, \dots, x_k}$.
$\fs(k)$ is the statement ``every function $f : [\Nb]^k \to \Nb$ has
an infinite set free for~$f$''. 
A function~$f : [\Nb]^{k+1} \to \Nb$ is \emph{stable}
if for every~$\sigma \in [\Nb]^k$, $\lim_s f(\sigma, s)$ exists.
$\sfs(k)$ is the restriction of $\fs(k)$ to stable functions.
$\fs$ is the statement~$(\forall k)\fs(k)$
\end{definition}

\begin{definition}[Thin set theorem]
Let $k \in \Nb$ and $f : [\Nb]^k \to \Nb$.
A set $A$ is \textit{thin for $f$} (or \emph{$f$-thin}) if $f([A]^n) \neq \Nb$.
$\ts(k)$ is the statement ``every function $f : [\Nb]^k \to \Nb$ has
an infinite set thin for~$f$''.
$\sts(k)$ is the restriction of $\ts(k)$ to stable functions.
$\ts$ is the statement~$(\forall k)\ts(k)$.
\end{definition}

Cholak \& al. studied extensively free set and thin set principles
in \cite{Cholak2001Free}, proving that $\fs(1)$ holds in $\rca$ while $\fs(2)$ does not,
$\fs(k+1)$ (resp. $\ts(k+1)$)
implies $\fs(k)$ (resp. $\ts(k)$) over $\rca$. They proved that
$\fs$ implies $\ts$ over $\rca$, and the more finely-grained result that $\fs(k)$ implies $\ts(k)$
and $\sfs(k)$ implies $\sts(k)$
over $\rca$ for every $k$. 

Some of the results where already stated by Friedman~\cite{FriedmFom:53:free}
without proof, notably there is an $\omega$-model of $\wkl$
which is not a model of $\ts(2)$, and $\aca$ does not imply $\ts$.
Cholak \& al. also proved that $\rca + \rt^k_2$ implies $\fs(k)$ for every~$k$
hence $\aca$ proves $\fs(k)$.
Wang showed in \cite{Wang2014Some} that neither $\fs$
nor $\ts$ implies $\aca$. He proved that $\rca \vdash \fs(k) \rightarrow \rrt^k_2$.
Rice~\cite{RiceThin} proved that $\sts(2)$ implies $\dnr$ over $\rca$.

We prove, using the equivalence between $\rrt^2_2$ and $\dnrzp$, that $\rca \vdash \ts(2) \imp \rrt^2_2$
and more generally $\rca \vdash \ts(k+1) \imp \dnr[\emptyset^{(k)}]$. We also prove that
$\sts(2)$ implies $\amt$ over $\rca$.

\bigskip

\subsection{Stable versions of the rainbow Ramsey theorem}

Consider a 2-bounded coloring $f$ of pairs as the history of interactions
between people in an infinite population. $f(x, s) = f(y, s)$ means that $x$ and $y$
interact at time $s$. In this world, $x$ and $y$ get \emph{married} if $f(x, s) = f(y, s)$
for cofinitely many $s$, whereas a person $x$ becomes a \emph{monk} if $f(x, s)$ is a fresh color
for cofinitely many $s$. Finally, a person $x$ is \emph{wise} if for each~$y$, 
either $x$ and~$y$ get married or $x$ and~$y$ eventually
break up forever, i.e., $(\forall y)[(\forall^\infty s) f(x,s) = f(y,s) \vee (\forall^\infty s)f(x,s) \neq f(y,s)]$. 
In particular married people and monks are wise.
Note that 2-boundedness implies that a person~$x$ can get married to at most one~$y$.

$\rrt^2_2$ states that given an world, we can find infinitely many instants where
people behave like monks. However we can weaken our requirement, leading to new principles.

\begin{definition}[Stable rainbow Ramsey theorem]
A coloring $f : [\Nb]^2 \to \Nb$ is 
\emph{rainbow-stable} if for every $x$, one of the following holds:
\begin{itemize}
	\item[(a)] There is a $y \neq x$ such that 
	$(\forall^\infty s)f(x, s) = f(y,s)$
	\item[(b)] $(\forall^{\infty} s) \card{\set{y \neq x : f(x, s) = f(y, s)}} = 0$
\end{itemize}
$\srrt^2_2$ is the statement ``every rainbow-stable 2-bounded coloring $f:[\Nb]^2 \to \Nb$
has a rainbow.''
\end{definition}

Hence in the restricted world of $\srrt^2_2$, everybody either gets married or becomes a monk.
$\srrt^2_2$ is a particular case of $\rrt^2_2$. It
is proven to have $\omega$-models with only low sets, hence is strictly weaker than $\rrt^2_2$.
Characterizations of $\rrt^2_2$ extend to $\srrt^2_2$ which is equivalent to diagonalizing
against any $\emptyset'$-computable total function, finding an infinite subset of a path
in a $\Delta^0_2$ tree of positive $\emptyset'$-computable measure, or being the subset
of an infinite set passing a Schnorr test relativized to~$\emptyset'$.

$\srrt^2_2$ happens to be useful as a factorization principle: It is strong enough to imply
principles like $\dnr$ or $\opt$ and weak enough to be consequence of many stable principles,
like $\srt^2_2$, $\sts(2)$ or $\semo$. It thus provides a factorization of the proofs
that $\ts(2)$ or $\emo$ both imply $\opt$ and $\dnr$ over $\omega$-models, which were proven independently
in \cite{Csima2009strength,RiceThin} for $\ts(2)$ and \cite{Lerman2013Separating} for $\emo$.

Wang used in \cite{Wang2014Cohesive} another version of stability for rainbow Ramsey theorems
to prove various results, like the existence of non-PA solution to any instance of $\rrt^3_2$.
This notion leads to a principle between $\rrt^2_2$ and $\srrt^2_2$.

\begin{definition}[Weakly stable rainbow Ramsey theorem]
A coloring $f : [\Nb]^2 \to \Nb$ is \emph{weakly rainbow-stable} if
$$
(\forall x)(\forall y)[(\forall^{\infty} s)f(x, s) = f(y,s) \vee (\forall^{\infty} s)f(x, s) \neq f(y,s)]
$$
$\wsrrt^2_2$ is the statement ``every weakly rainbow-stable 2-bounded coloring $f:[\Nb]^2 \to \Nb$
has an infinite rainbow.''
\end{definition}

Weak rainbow-stability can be considered as the ``right'' notion of stability for 2-bounded colorings
as one can extract an infinite weakly rainbow-stable restriction of any 2-bounded coloring using
cohesiveness.

However the exact strength of $\wsrrt^2_2$ is harder to tackle. A characterization
candidate would be computing an infinite subset of a path in a $\emptyset'$-computably
graded $\Delta^0_2$ tree where the notion of computable gradation is taken from the restriction
of Martin-L\"of tests to capture \emph{computable} random reals. We prove that it is enough be able to escape finite
$\Delta^0_2$ sets to prove $\wsrrt^2_2$. We also separate $\wsrrt^2_2$
from~$\rrt^2_2$ by proving that $\wsrrt^2_2$ contains an $\omega$-model with only low sets.
The question of exact characterizations of $\wsrrt^2_2$ remains open.

Due to the lack of characterizations of $\wsrrt^2_2$, only $\sfs(2)$ is proven to be strong enough
to imply $\wsrrt^2_2$ among $\sfs(2)$, $\sts(2)$ and $\semo$.

\bigskip

\subsection{Notation}

The set of finite binary strings is denoted by $2^{<\Nb}$. We write~$\epsilon$ for the empty string.
The length of $\sigma \in 2^{<\Nb}$ is denoted $|\sigma|$. For $i \in \N$, and $\sigma \in \str$, $\sigma(i)$ is the $(i+1)$-th bit of $\sigma$. For $\sigma,\tau \in \str$, we say that $\sigma$ is a prefix of $\tau$ (written $\sigma \preceq \tau$) if $|\sigma| \leq |\tau|$ and $\sigma(i)=\tau(i)$ for all $i<|\sigma|$. Given a finite string $\sigma$, 
$\Gamma_\sigma = \{\tau \in 2^{<\Nb} : \sigma \preceq \tau \}$.

We denote by $2^{\Nb}$ the space of infinite binary sequences. We also refer to the elements of $2^{\Nb}$ as \emph{sets (of integers)}, as any $X \subseteq  \N$ can be identified with its characteristic sequence, which is an element of $2^{\Nb}$. For a string $\sigma$, $\cyl{\sigma}$ is the set of $X \in \cs$ whom $\sigma$ is a prefix of.

A \emph{binary tree} $T$ is a subset of $2^{<\Nb}$ downward closed under prefix relation. 
Unless specified otherwise we will consider only binary trees.
A sequence $P$ is a \emph{path} of $T$ if any initial segment of $P$ is in $T$. We denote by $[T]$ the $\Pi^0_1$
class of paths through $T$.

Given a set $X$ and an element $a$, we write $a < X$ to state that $a$ is strictly below each member of $X$.
We denote by $\Gamma^i_X$ the set $\{ \tau \in 2^{<\Nb} : (\forall s < \card{\tau}) s \in X \imp \tau(s) = i \}$.
$\Gamma_X = \Gamma^0_X \cup \Gamma^1_X$. Whenever $X = \{n\}$, we shall write $\Gamma^i_n$ for $\Gamma^i_{\{n\}}$.

\section{Rainbow Ramsey theorem}

The computational strength of the rainbow Ramsey theorem for pairs is well understood,
thanks to its remarkable connections with algorithmic randomness, and more precisely
the notion of diagonal non-computability.

\begin{definition}[Diagonal non-computability]
A function $f : \Nb \to \Nb$ is \emph{diagonally non-computable} relative to $X$ if $(\forall e)f(e) \neq \Phi^X_e(e)$. 
A function $f : \Nb \to \Nb$ is \emph{fixpoint-free} relative to $X$
if $(\forall e)W_e \neq W_{f(e)}$.
$\dnr$ (resp. $\fpf$) is the statement ``For every $X$, there exists a function d.n.c. (resp. f.p.f.) relative to~$X$''.
For every $n \in \Nb$, $\dnr[0^{(n)}]$ is the statement ``For every $X$, there exists a function
d.n.c. relative to~$X^{(n)}$''.
\end{definition}

It is well-known that fixpoint-free degrees are precisely d.n.c. degrees, and that this equivalence
holds over $\rca$. Hence $\rca \vdash \dnr \biimp \fpf$.
Miller~\cite{MillerPersonal} gave a characterization of d.n.c. degrees relative to $\emptyset'$:

\begin{theorem}[Miller \cite{MillerPersonal}]\label{thm-rrt22-dnrzp}
$\rca \vdash \rrt^2_2 \biimp \dnrzp$
\end{theorem}

A first consequence of Theorem~\ref{thm-rrt22-dnrzp} 
is another proof of $\rca + \ran{2} \vdash \rrt^2_2$.
Moreover it will enable us to prove
a lot of implications from other principles to $\rrt^2_2$ -- Theorem~\ref{thm:em-dnrzp}, Theorem~\ref{thm:ts2-dnrzp} --.
The author, together with Bienvenu and Shafer defined in~\cite{Bienvenu2014role} a property over $\omega$-structures, the No Randomized Algorithm
property, and classified a wide range of principles depending on whether their $\omega$-models have this property.
They proved that for any principle $P$ having this property, there exists an $\omega$-model of $\rca + \ran{2}$
which is not a model of~$P$. In particular, there exists an $\omega$-model of $\rca + \rrt^2_2$ which is not a model of~$P$.

A careful look at the proof of Theorem~\ref{thm-rrt22-dnrzp} gives the following relativized version:
\begin{theorem}[Miller \cite{MillerPersonal}, $\rca$]\label{thm:rrt22-dnrzp-rel}
Fix a set $X$. 
\begin{itemize}
  \item There is an $X$-computable 2-bounded coloring $f : [\Nb]^2 \to \Nb$
such that every infinite $f$-rainbow computes (not relative to $X$) a function d.n.c. relative to $X^{'}$.
  \item For every $X$-computable 2-bounded coloring $f : [\Nb]^2 \to \Nb$
  and every function $g$ d.n.c. relative to $X'$, there exists a $g \oplus X$-computable infinite $f$-thin set.
\end{itemize}
\end{theorem}

Theorem~\ref{thm:rrt22-dnrzp-rel} can be generalized by a straightforward adaptation
of \cite[Theorem~2.5]{Csima2009strength}. We first state a technical lemma.

\begin{lemma}[$\rca$]\label{lem:rrtn-jump}
Fix a standard $n \geq 1$ and $X \subseteq \Nb$.
For every $X'$-computable 2-bounded coloring $f : [\Nb]^n \to \Nb$
there exists an $X$-computable 2-bounded coloring $g : [\Nb]^{n+1} \to \Nb$
such that every rainbow for $g$ is a rainbow for $f$.
\end{lemma}
\begin{proof}
Using the Limit Lemma, there exists an $X$-computable approximation function $h: [\Nb]^{n+1} \to \Nb$
such that $\lim_s h(\vec{x}, s) = f(\vec{x})$ for every $\vec{x} \in [\Nb]^n$.

Let~$\tuple{ \dots }$ be a standard coding of the lists of integers into~$\Nb$
and~$\prec_\Nb$ be a standard total order over~$\Nb^{<\Nb}$.
We define an $X$-computable 2-bounded coloring $g:[\Nb]^{n+1}$ as follows.
$$
g(\vec{x}, s) = \cond{
   \tuple{h(\vec{x}, s), s, 0} & 
   \mbox{ if there is at most one } \vec{y} \prec_\Nb \vec{x} \mbox{ s.t. } h(\vec{y},s) = h(\vec{x}, s)\\
   \tuple{rank_{\prec_\Nb}(\vec{x}), s, 1} & \mbox{ otherwise}
}
$$
(where $rank_{\prec_\Nb}(\vec{x})$ is the position of $\vec{x}$
for any well-order $\prec_\Nb$ over tuples). 

By construction $g$ is 2-bounded and $X$-computable.
We claim that every infinite rainbow for $g$ is a rainbow for $f$.

Let $A$ be an infinite rainbow for $g$. Assume for the sake of contradiction that $\vec{x}, \vec{y} \in [A]^n$
are such that $\vec{y} \prec_\Nb \vec{x}$ and $f(\vec{y}) = f(\vec{x})$.
Fix $t \in \Nb$ such that $h(\vec{z}, s) = f(\vec{z})$ whenever $\vec{z} \preceq_\Nb \vec{x}$
and $s \geq t$. Fix $s$ such that $s \in A$, $s \geq t$ and $s > max(\vec{x})$. Notice that 
since $f$ is 2-bounded and $h(\vec{z}, s) = f(\vec{z})$ for every $\vec{z} \preceq_\Nb \vec{x}$,
we have $g(\vec{z}, s) = \tuple{h(\vec{z}, s), s, 0} = \tuple{f(\vec{z}), s, 0}$ for every $\vec{z} \preceq_\Nb \vec{x}$.
Hence
$$
g(\vec{x}, s) = \tuple{f(\vec{x}), s, 0} = \tuple{f(\vec{y}), s, 0} = g(\vec{y}, s)
$$
contradicting the fact that $A$ is a rainbow for $g$.
\end{proof}

We can now deduce several relativizations of some existing results.

\begin{theorem}[$\rca$]\label{thm:rrtk-dnrzk}
For every standard $n \geq 1$ and $X \subseteq \Nb$,
there is an $X$-computable 2-bounded coloring function $f : [\Nb]^{n+1} \to \Nb$
such that every infinite rainbow for $c$ computes (not relative to $X$) a function d.n.c. relative to $X^{(n)}$.
\end{theorem}
\begin{proof}
By induction over $n$.
Case $n=1$ is exactly the statement of Theorem~\ref{thm:rrt22-dnrzp-rel}.
Assume it holds for some $n \in \Nb$. Fix an $X'$-computable
2-bounded coloring $g : [\Nb]^n \to \Nb$ such that every infinite rainbow for $g$ computes
a function d.n.c. relative to $(X^{(n-1)})' = X^n$.
By Lemma~\ref{lem:rrtn-jump} there exists an $X$-computable coloring $f: [\Nb]^{n+1} \to \Nb$
such that every infinite rainbow for $f$ is a rainbow for $g$.
\end{proof}

\begin{corollary}\label{cor:rrtk-dnrzk}
For every standard $k \geq 1$, $\rca \vdash \rrt^{(k+1)}_2 \imp \dnr[\emptyset^{(k)}]$.
\end{corollary}

The other direction does not hold. In fact, for every standard $k$,
there exists an $\omega$-model of $\dnr[\emptyset^{(k)}]$ not model of $\rrt^3_2$ as we will see later (Remark~\ref{rem:dnrzk-rrtk}).

\begin{definition}[Hyperimmunity]\ 
A function $h:  \Nb \to \Nb$ \emph{dominates} a function $g : \Nb \to \Nb$ if $h(n) > g(n)$ for all but finitely many $n \in \Nb$.
The \emph{principal function}~$p_A$ of a set~$A = \{x_0 < x_1 < \dots \}$ is defined by~$p_A(n) = x_n$ for every~$n \in \Nb$.
Given a set $X$, a set $A$ is \emph{hyperimmune relative to $X$} if its principal function $p_A$ is not dominated
by any $X$-computable function.
$\hyp$ is the statement ``For every set~$X$, there exists a set $Y$ hyperimmune relative to~$X$''.
\end{definition}

\begin{theorem}[$\rca$]\label{thm:rrtk-opt}
For every standard $n \geq 1$ and $X \subseteq \Nb$,
there is an $X$-computable 2-bounded coloring function $f : [\Nb]^{n+2} \to \Nb$
such that every infinite rainbow for $f$ is a set hyperimmune relative to $X^{(n)}$.
\end{theorem}
\begin{proof}
As usual, by induction over $n$.
Case $n=1$ is exactly the statement of Theorem~4.1 of \cite{Csima2009strength}.
Assume it holds for some $n \in \Nb$. Fix an $X'$-computable
2-bounded coloring $g : [\Nb]^{n+1} \to \Nb$ such that every infinite rainbow for $g$ is a set hyperimmune
relative to $(X^{(n-1)})' = X^n$.
By Lemma~\ref{lem:rrtn-jump} there there exists an $X$-computable coloring $f: [\Nb]^{n+2} \to \Nb$
such that every infinite rainbow for $f$ is a rainbow for $g$. This concludes the proof.
\end{proof}

A simple consequence is that any $\omega$-model of $\rrt^3_2$
is a model of $\amt$. We will see later by a more careful analysis that $\rca \vdash \rrt^3_2 \imp \sts(2)$
and $\rca \vdash \sts(2) \imp \amt$. Bienvenu et al.\ proved in~\cite{Bienvenu2014role} 
the existence of an $\omega$-model of $\rrt^2_2$ not model of $\amt$.

\begin{remark}\label{rem:rt-hif}
This theorem is optimal in the sense that every computable 2-bounded coloring $c : [\Nb]^{n+1} \to \Nb$
has an infinite rainbow of hyperimmune-free degree relative to $0^{(n)}$ by combining
a theorem from Jockusch~\cite{Jockusch1972Ramseys} and the relativized version of the hyperimmune-free basis theorem.
\end{remark}

As $\srt^2_2$ and $\rrt^2_2$ are both consequences of $\rt^2_2$ over $\rca$, one might wonder
how they do relate each other. The answer is that they are incomparable as states Corollary \ref{cor-srt-rrt}
and Csima \& Mileti in~\cite{Csima2009strength}.

The first direction is a consequence of a very tricky proof of separation
of $\rt^2_2$ and $\srt^2_2$ using non-standard models. This separation question
was a long-standing open question, recently positively answered by Chong et al.\ in~\cite{Chong2014metamathematics}.

\begin{theorem}[Chong et al.\ \cite{Chong2014metamathematics}]
There is a model of $\rca + \bst + \neg \ist + \srt^2_2$ having only $\Delta^0_2$ (in fact low) sets.
\end{theorem}

\begin{theorem}
There is no model of $\rca + \rrt^2_2$ having only $\Delta^0_2$ sets.
\end{theorem}
\begin{proof}
Using the characterization of $\rrt^2_2$ by $\dnrzp$ proven by Joe Miller -- Theorem \ref{thm-rrt22-dnrzp} --,
let $f$ be a diagonally non-computable function relative to $\emptyset'$.
If $f$ were $\Delta^0_2$, then letting $e$ be a Turing index such that $\Phi^{\emptyset'}_e = f$,
we would have $f(e) \neq \Phi^{\emptyset'}_e(e)$, contradiction.
\end{proof}

\begin{corollary}\label{cor-srt-rrt}
$\rca + \srt^2_2 \not \vdash \rrt^2_2$
\end{corollary}

The question of separating $\rt^2_2$ and $\srt^2_2$ in $\omega$-models is still an open question.
A related question whose positive answer would give a separation of $\rt^2_2$ from $\srt^2_2$ is the following:

\begin{question}
Is there an $\omega$-model of $\rca + \srt^2_2$ which not a model of $\rrt^2_2$~?
\end{question}

Even if the question is stronger, this approach could be simpler as $\rrt^2_2$ coincide with $\dnrzp$
which admits a set complete for the corresponding Muchnik degree, ie any set d.n.c. relative to $\emptyset'$.

Csima \& Mileti proved in~\cite{Csima2009strength} that there exists a computable 2-bounded
coloring $c : [\Nb]^2 \to \Nb$ such that every infinite set thin for $c$ computes
a set of hyperimmune degree. We now give an alternative proof of the same statement using $\Pi^0_1$-genericity.

\begin{definition}
A set $X$ is \emph{$\Pi^0_1$-generic} if for all $\Sigma^0_2$ classes $G$,
either $X$ is in $G$ or there is a $\Pi^0_1$ class $F$ disjoint from $G$ such that $X$ is in $F$.
\end{definition}

\begin{theorem}[Monin in \cite{Monin2014Higher}]\label{thm:char-pi01gen}
A set $X$ is $\Pi^0_1$-generic iff it is of hyperimmune-free degree.
\end{theorem}

\begin{theorem}\label{thm:dnrzp-hyp}
No $\Pi^0_1$-generic set computes a function d.n.c. relative to $\emptyset'$.
\end{theorem}
\begin{proof}
Fix any functional $\Psi$.
Consider the $\Sigma^0_2$ class
$$
U = \set{X \in \cs : (\exists e)[\Psi^X(e) \uparrow \vee \Psi^X(e) = \Phi^{\emptyset'}_e(e)]}
$$

Consider any $\Pi^0_1$-generic $X$ such that $\Psi^X$ is total.
Either $X \in U$, in which case $\Psi^X(e) = \Phi_e^{\emptyset'}(e)$ hence $\Psi^X$ is not d.n.c. relative to $\emptyset'$.
Or there exists a $\Pi^0_1$ class $F$ disjoint from $U$ and containing $X$. Any member of $F$ computes a function d.n.c. relative to $\emptyset'$.
In particular any $\Delta^0_2$ set of PA degree computes such a function, contradiction.
\end{proof}

\begin{corollary}
Every function d.n.c. relative to $\emptyset'$ is of hyperimmune degree.
\end{corollary}
\begin{proof}
Thanks to Theorem~\ref{thm:char-pi01gen} we can restate Theorem~\ref{thm:dnrzp-hyp}
as \emph{no hyperimmune-free set computes a function d.n.c. relative to $\emptyset'$},
hence every such function is of hyperimmune degree.
\end{proof}

\section{The Erd\H{o}s-Moser theorem}
Bovykin and Weiermann~\cite{Bovykin2005strength} decomposed
$\rt^2_2$ into $\emo$ and $\ads$ as follows:
Given a coloring~$f : [\Nb]^2 \to 2$, we can see~$f$ as a tournament~$T$
such that whenever~$x <_\Nb y$, $T(x,y)$ holds if and only if~$f(x,y) = 1$.
Any transitive subtournament~$H$ can be seen as a linear order $(H, \prec)$
such that whenever~$x <_\Nb y$, ~$x \prec y$ if and only if~$f(x,y) = 1$.
Any infinite ascending or descending sequence is $f$-homogeneous.
This decomposition also holds for the stable versions and enables us
to make~$\semo$ inherit from several properties of~$\srt^2_2$.

Many principles in reverse mathematics are~$\Pi^1_2$ statements~$(\forall X)(\exists Y)\Phi(X,Y)$,
where $\Phi$ is an arithmetic formula. They usually come with a natural collection of \emph{instances}~$X$.
A set~$Y$ such that~$\Phi(X,Y)$ holds is a \emph{solution} to~$X$.
For example, in Ramsey's theorem for pairs, and instance is a coloring~$f : [\Nb]^2 \to 2$
and a solution to~$f$ is an infinite $f$-homogeneous set.
Many proofs of implications~$\Qsf \imp \Psf$ over~$\rca$ happen to be \emph{computable reductions} from $\Psf$ to~$\Qsf$.

\begin{definition}[Computable reducibility]
Fix two~$\Pi^1_2$ statements~$\Psf$ and~$\Qsf$.
We say that $\Psf$ is \emph{computably reducible} to~$\Qsf$ (written $\Psf \leq_c \Qsf$)
if every~$\Psf$-instance~$I$ computes a~$\Qsf$-instance~$J$ such that for every solution~$X$ to~$J$,
$X \oplus I$ computes a solution to~$I$.
\end{definition}

A computable reducibility~$\Psf \leq_c \Qsf$ can be seen as a degenerate case of an implication~$\Qsf \imp \Psf$
over~$\omega$-models, in which the principle~$\Qsf$ is applied at most once.
In order to prove that~$\Psf \not \leq_c \Qsf$, it suffices to construct one $\Psf$-instance $I$
such that for every $I$-computable $\Qsf$-instance~$J$, there exists some solution~$X$ to~$J$
such that~$X \oplus I$ does not compute a solution to~$I$.
We need the following stronger notion of avoidance which implies in particular computable non-reducibility.

\begin{definition}[Avoidance]
Let $\Psf$ and $\Qsf$ be $\Pi^1_2$ statements. $\Psf$ is \emph{$\Qsf$-avoiding} if
for any set $X$ and any $X$-computable instance $I$ of $\Qsf$ having no $X$-computable solution,
any $X$-computable instance of $\Psf$ has a solution $S$ such that $I$ has no $X \oplus S$-computable solution.
\end{definition}

\begin{example}
Hirschfeldt and Shore proved in~\cite{Hirschfeldt2007Combinatorial}
that if some set~$X$ does not compute an infinite d.n.c.\ function,
every $X$-computable linear order has an infinite ascending or descending sequence~$Y$
such that~$X \oplus Y$ does not compute an infinite d.n.c.\ function.
Therefore~$\ads$ is~$\dnr$-avoiding.

On the other side, the author~\cite{PateyCombinatorial} showed
the existence of an infinite computable binary tree $T \subseteq 2^{<\Nb}$ with no infinite, computable path,
together with a computable coloring~$f : [\Nb]^2 \to 2$ such that every infinite $f$-homogeneous
set computes an infinite path through~$T$. Therefore $\rt^2_2$ is not $\wkl$-avoiding.
\end{example}

\begin{theorem}\label{thm:avoidance-vs-semo}
If $\Psf \leq_c \srt^2_2$ and $\sads$ is $\Psf$-avoiding, then $\Psf \leq_c \semo$.
\end{theorem}
\begin{proof}
Let $I$ be any instance of $\Psf$. As $\Psf \leq_c \srt^2_2$, there exists
an $I$-computable stable coloring $f : [\Nb]^2 \to 2$ such that for any infinite $f$-homogeneous set $H$,
$I \oplus H$ computes a solution to $I$. The coloring $f$ can be seen as a tournament $T$
where for each~$x < y$, $T(x, y)$ holds iff $f(x, y) = 1$. If $T$ has an infinite sub-tournament
$U$ such that $I \oplus H$ does not compute a solution to $I$, consider $H$ as an $I \oplus H$-computable
stable linear order. Then by $\Psf$-avoidance of $\sads$, there exists a solution $S$ to $H$
such that $I \oplus H \oplus S$ does not compute a solution to $I$. But $S$ is an infinite $f$-homogeneous set,
contradicting our choice of $f$.
\end{proof}

\begin{corollary}[Kreuzer \cite{Kreuzer2012Primitive}]
There exists a transitive computable tournament having no low infinite subtournament.
\end{corollary}
\begin{proof}
Consider the principle $\overline{\mathsf{Low}}$ stating ``$\forall X \exists Y (Y \oplus X)' \not \leq_T X'$''.
Downey et al.\ proved in~\cite{Downey20010_2} that for every set~$X$, there exists an~$X$-computable instance of~$\srt^2_2$ with no solution low over~$X$. In other words, $\overline{\mathsf{Low}} \leq_c \srt^2_2$. On the other side,
Hirschfeldt et al.\ \cite{Hirschfeldt2007Combinatorial} proved that every linear order~$L$ of order type~$\omega+\omega^{*}$ has 
an infinite ascending or descending sequence which is low over~$L$. Therefore~$\sads$ is $\overline{\mathsf{Low}}$-avoiding.
By Theorem~\ref{thm:avoidance-vs-semo}, $\overline{\mathsf{Low}} \leq_c \semo$.
\end{proof}

\begin{corollary}\label{cor:omega-sem-dnr}
Any $\omega$-model of $\semo$ is a model of $\dnr$.
\end{corollary}
\begin{proof}
Hirschfeldt et al.\ proved in~\cite{Hirschfeldt2008strength} that $\dnr \leq_c \srt^2_2$
and in \cite{Hirschfeldt2007Combinatorial} that $\ads$ is $\dnr$-avoiding.
By Theorem~\ref{thm:avoidance-vs-semo}, $\dnr \leq_c \semo$.
\end{proof}

\begin{corollary}
There exists an $\omega$-model of $\cac$ which is not a model of $\semo$.
\end{corollary}
\begin{proof}
Hirschfeldt et al.\ constructed in~\cite{Hirschfeldt2007Combinatorial}
an $\omega$-model of~$\cac$ which is not a model of~$\dnr$.
\end{proof}

\begin{corollary}
$\coh \leq_c \srt^2_2$ if and only if $\coh \leq_c \semo$
\end{corollary}
\begin{proof}
The author associated in~\cite{Patey2015Dissent}
a $\Pi^{0,\emptyset'}_1$ class $\Ccal(\vec{R})$
to any sequence of sets~$R_0$, $R_1$, ..., so that
a degree bounds an~$\vec{R}$-cohesive set if and only if
its jump bounds a member of~$\Ccal(\vec{R})$.
Hirschfeldt et al.~\cite{Hirschfeldt2007Combinatorial} proved that
every $X$-computable instance $I$ of~$\sads$ has a solution $Y$ low over~$X$.
Therefore, if $X$ does not compute an~$\vec{R}$-cohesive set,
then~$X'$ does not compute a member of~$\Ccal(\vec{R})$.
As~$(Y \oplus X)' \leq X'$, $(Y \oplus X)'$ does not compute a member of~$\Ccal(\vec{R})$,
$Y \oplus X$ does not compute an~$\vec{R}$-cohesive set.
In other words, $\sads$ is~$\coh$-avoiding. Conclude by Theorem~\ref{thm:avoidance-vs-semo}.
\end{proof}

\begin{definition}
Let $T$ be a tournament on a domain $D \subseteq \N$. A $n$-cycle is
a tuple $\{x_1, \dots, x_n\} \in D^n$ such that 
for every $0 < i < n$, $T(x_i, x_{i+1})$ holds
and $T(x_n, x_1)$ holds.
\end{definition}

Kang~\cite{Kang2014Combinatorial} attributed to Wang a direct proof of $\rca \vdash \emo \imp \rrt^2_2$.
We provide an alternative proof using the characterization of $\rrt^2_2$
by $\dnrzp$ from~Miller.

\begin{theorem}\label{thm:em-dnrzp}
$\rca \vdash \emo \rightarrow \dnrzp$
\end{theorem}
\begin{proof}

Let $X$ be a set. Let $g(.,.)$ be a total $X$-computable function such that $\Phi_e^{X'}(e)=\lim_s g(e,s)$ if the limit exists, and $\Phi^{X'}_e(e) \uparrow$ if the limit does not exist. Interpret $g(e,s)$ as the code of a finite set $D_{e,s}$ of size $3^{e+1}$. We define the tournament~$T$ by $\Sigma_1$-induction as follows. Set $T_0=\emptyset$. At stage~$s+1$, do the following. Start with $T_{s+1}=T_s$. Then, for each $e<s$, take the first pair $\{x,y\} \in D_{e,s} \setminus \bigcup_{k<e} D_{k,s}$ (notice that such a pair exists by cardinality assumptions on the $D_{e,s}$), and if $T_{s+1}(s,x)$ and $T_{s+1}(s,y)$ are not already assigned, assign them in a way that $\{x,y,s\}$ forms a $3$-cycle in $T_{s+1}$. Finally, for any $z<s$ such that $T_{s+1}(s,z)$ remains undefined, assign any truth value to it in a predefined way (e.g., for any such pair $\{x,y\}$, set $T_{s+1}(x,y)$ to be true if $x<y$, and false otherwise). This finishes the construction of $T_{s+1}$. Set $T=\bigcup_s T_s$, which must exist as a set by $\Sigma_1$-induction. 

First of all, notice that $T$ is a tournament of domain $[\N]^2$, as at the end of stage~$s+1$ of the construction $T(x,y)$ is assigned a truth value for (at least) all pairs $\{x,y\}$ with $x<s$ and $y<s$. By $\emo$, let $T'$ be a transitive subtournament of~$T$ of infinite domain~$A$. Let $f(e)$ be the code of the finite set $A_e$ consisting of the first $3^{e+1}$ elements of~$T'$. We claim that $f(e) \not=\Phi_e^{X'}(e)$ for all~$e$, which would prove $\dnrzp$. Suppose otherwise, i.e., suppose that $\Phi_e^{X'}(e)=f(e)$ for some~$e$. Then there is a stage $s_0$ such that $f(e)=g(e,s)$ for all~$s \geq s_0$ or equivalently $D_{e,s}=A_e$ for all~$s \geq s_0$. Let $N_e=\max(A_e)$. We claim that for any $s$ be bigger than both $\max(\bigcup_{e,s < N_e} D_{e,s})$ and $s_0$, the restriction of $T$ to $A_e \cup \{s\}$ is not a transitive subtournament, which contradicts the fact that the restriction $T'$ of $T$ to the infinite set $A$ containing $A_e$ is transitive. 

To see this, let $s$ be such a stage. At that stage~$s$ of the construction of $T$, a pair $\{x,y\} \in D_{e,s} \setminus \bigcup_{k<e} D_{k,s}$ is selected, and since $D_{e,s}=A_e$, this pair is contained in $A_e$. Furthermore, we claim that $T(s,x)$ and $T(s,y)$ become assigned at that precise stage, i.e., were not assigned before. This is because, by construction of $T$, when the value of some $T(a,b)$ is assigned at a stage~$t$, either $a \leq t$ or $b \leq t$. Thus, if $T(s,x)$ was already assigned at the beginning of stage~$s$, it would have in fact been assigned during or before stage~$x$. However, $x \in A_e$, so $x < N_e$, and at stage~$N_e$ the number $s$, by definition of $N_e$, has not appeared in the construction yet. In particular $T(s,x)$ is not assigned at the end of stage~$x$. This proves our claim, therefore $T(s,x)$ and $T(s,y)$ do become assigned exactly at stage~$s$, in a way -- still by construction -- that $\{x,y,s\}$ form a $3$-cycle for~$T$. Therefore the restriction of $T$ to $A_e \cup \{s\}$ is not a transitive subtournament, which is what we needed to prove. 
\end{proof}

\begin{corollary}[Wang in~\cite{Kang2014Combinatorial}]\label{cor:emo-rrt22}
$\rca \vdash \emo \to \rrt^2_2$
\end{corollary}
\begin{proof}
Immediate by Theorem~\ref{thm:em-dnrzp} and Theorem~\ref{thm-rrt22-dnrzp}.
\end{proof}

\begin{corollary}
$\semo$ does not imply~$\emo$ over~$\rca$.
\end{corollary}
\begin{proof}
Immediate by Theorem~\ref{thm:em-dnrzp},
Theorem~\ref{thm-rrt22-dnrzp} and Corollary~\ref{cor-srt-rrt}
\end{proof}

We have seen (see Corollary~\ref{cor:omega-sem-dnr})
that every $\omega$-model of~$\semo$ is a model of~$\dnr$.
We now give a direct proof of it and show that it holds over~$\rca$.

\begin{theorem}
$\rca \vdash \semo \rightarrow \dnr$
\end{theorem}
\begin{proof}
This is obtained by small variation of the proof of Theorem~\ref{thm:em-dnrzp}.
Fix a set $X$. Let $g(.,.)$ be a total $X$-computable function such that 
$\Phi^X_e(e)=\lim_s g(e,s)$ if~$\Phi^X_e(e) \downarrow$
and~$\lim_s g(e,s) = 0$ otherwise. Interpret $g(e,s)$ as a code
of a finite set $D_{e,s}$ of size $3^{e+1}$ such that~$min(D_{e,s}) \geq e$ and construct the infinite tournament~$T$
accordingly. The argument for constructing a function d.n.c.\ relative to~$X$
given an infinite transitive subtournament is similar. We will only prove that
the tournament~$T$ is stable.

Fix some~$u \in \Nb$. 
By $\bst$, which is provable from $\semo$ over $\rca$ (see \cite{Kreuzer2012Primitive}),
there exists some stage~$s_0$ after which~$D_{e,s}$ remains constant
for every~$e \leq u$. If~$u$ is part of a pair~$\{x,y\} \subset D_{e,s}$
for some~$s \geq s_0$ and~$e$, then~$e \leq u$ because $min(D_{e,s}) \geq e$.
As the~$D_{e,s}$'s remain constant for each~$e \leq u$, 
the pair~$\{x,y\}$ will be chosen at every stage~$s \geq s_0$
and therefore~$T(u, s)$ will be assigned the same value for every~$s \geq s_0$.
If~$u$ is not part of a pair~$\{x,y\}$, it will always be assigned
the default value at every stage~$s \geq s_0$.
In both cases, $T(u,s)$ stabilizes at stage~$s_0$.
\end{proof}

\section{Free set and thin set theorems}

Some priority or forcing constructions involving $\srt^2_2$ split their requirements
by color and do not exploit the fact that there exists only two colors.
For example the absence of universal instance for principles between~$\rt^2_2$ and~$\srt^2_2$
proven by Mileti in~\cite[Theorem~5.4.2]{Mileti2004Partition} has been generalized by the author to principles between~$\rt^2_2$
and~$\sts(2)$ in~\cite{Patey2015Degrees}. The separation of~$\emo$ from~$\srt^2_2$ by Lerman et al.~\cite{Lerman2013Separating}
has been adapted to a separation of~$\emo$ from~$\sts(2)$ as well (see~\cite{Patey2013note}).

\begin{question}
Does $\fs(2)$ imply $\emo$ (or even $\semo$) over $\rca$ ?
\end{question}

The following question is still open:
\begin{question}
Is there any $k$ such that $\rca \vdash \ts(k) \rightarrow \fs(k)$ ?
\end{question}
Cholak et al.\ conjectured that it is never the case.

\begin{lemma}\label{lem-srt-sfs}
$\rca \vdash \srt^2_2 \rightarrow \sfs(2)$
\end{lemma}
\begin{proof}
We adapt the proof of~\cite[Theorem~5.2]{Cholak2001Free}.
Let $f : [\Nb]^2 \to \Nb$ be a stable coloring function over pairs.
For $\vec w$ an ordered $k$-tuple and $1 \leq j \leq k$ we write $(\vec w)_j$
for the $j$th component of $\vec w$.
Define
$$
S = \set{(x,y) \in \Nb^2 : f(x,y) < y \wedge f(x,y) \not \in \set{x,y}}
$$
Given some $\vec x \in S$, 
let $i(\vec x)$ be the least $j$ such that $f(\vec x) < (\vec x)_j$. Such a $j$
exists because $\vec x \in S$.
Let $h(\vec x)$ be the increasing ordered pair obtained from $\vec x$
by replacing $(\vec x)_{i(\vec x)}$ by $f(\vec x)$. Note that~$h(\vec x)$ is lexicographically
smaller than~$\vec x$.
Let $c(\vec x)$ be the least $j \in \Nb$ such that $h^{(j)}(\vec x) \not \in S$ or $i(h^{(j)}(\vec x)) \neq i(\vec x)$
where $h^{(j)}$ is the $j$th iteration of $h$. The function $c$ is well-defined because the lexicographic order
is a well-order. Define a function $g : [\Nb]^2 \to 6$ as follows for each~$x < y$:
$$
g(x,y) = \cond{
  0 & \mbox{ if } f(x,y) \in \set{x,y}\\
  1 & \mbox{ if } f(x,y) > y\\
  2i(x,y) + j & \mbox{ if } (x,y) \in S, j \leq 1 \mbox{ and } c(x,y) \equiv j \mbox{ mod } 2
}
$$
Fix an $x$. Because $f$ is stable there is a $y_0$ such that for every $y \geq y_0$ $f(x,y) = f(x,y_0)$.

  \emph{Case 1}: If there is a $y_1$ such that $f(x,y_1) \in \set{x,y_1}$ then for every
$y, w > max(y_0,y_1)$, $f(x,y) \in \set{x,y}$ iff $f(x,w) \in \set{x,w}$ and hence
after a threshold first condition will either be always fulfilled or will never be.
  
  \emph{Case 2}: For every $y \geq max(y_0, f(x,y_0))$, $f(x,y) = f(x,y_0) \leq y$.
  Hence second condition will be fulfilled for finitely many $y$.
  
  \emph{Case 3}: It suffices to check that $i$ and $c$ are stable when $f$ is.
  If $f(x,y_0) < x$ then $i(x,y) = 1$ for every $y \geq y_0$. If $x \leq f(x,y_0) < y_0$
  then $x \leq f(x,y) < y_0 \leq y$ for every $y \geq y_0$. Hence $i$ is stable.
  It remains to check stability of $c(x,y)$. By induction over $x$:
  \begin{itemize}
  \item If $f(x,y_0) < x$ then $h(x,y) = (f(x,y_0), y)$ for every $y \geq y_0$.
  By stability of $f$, there is a $y_1$ such that $f(f(x,y_0),y) = f(f(x,y_0), y_1)$
  for every $y \geq y_1$. For $y > max(y_1, f(f(x,y_0), y_1))$, $f(f(x,y_0), y) = f(f(x,y_0),y_1) < y$.
  If $f(f(x,y_0),y_1) = f(x,y_0)$ then $(f(x,y_0),y_1) = h(x,y) \not \in S$ hence $j = 1$ for every 
  $y > max(y_1, f(f(x,y_0), y_1))$.
  Otherwise $h(x,y) \in S$. If $f(h(x,y)) = f(f(x,y_0), y)\allowbreak = f(f(x,y_0), y_1) > f(x,y_0)$ then 
  $i(h(x,y)) \neq i(x,y)$ and $j = 1$ for every $y > max(y_1, f(f(x,y_0), y_1))$. Otherwise
  $h(x,y) \in S$ and $i(h(x,y) = i(x,y)$ so $j = 1 + i$ where $i$ is the least integer such that
  $h^{(i)}(f(x,y_0), y) \in S$ or $i(h^{(i)}(f(x,y_0), y) \neq i(x,y) = i(h(x,y)$.
  Hence $j = 1 + c(f(x,y_0), y)$. By induction hypothesis, there is a $y_2$ such that
  for every $y \geq y_2$, $c(f(x,y_0), y) = c(f(x,y_0),y_2)$. So for every $y, w > max(y_1, y_2, f(f(x,y_0), y_1))$
  $c(x,y) = c(x, w)$.
  \item If $x \leq f(x,y_0) < y_0$ then for every $y \geq y_0$, $h(x,y) = (x,f(x,y_0))$
  and hence $c(x,y) = c(x,y_0)$.
  \end{itemize} 
\end{proof}

\begin{corollary}\label{cor-srt-sts}
$\rca \vdash \srt^2_2 \rightarrow \sts(2)$
\end{corollary}
\begin{proof}
Apply Lemma~\ref{lem-srt-sfs} using the restriction of 
Theorem 3.2 to stable functions in~\cite{Cholak2001Free}.
\end{proof}

\begin{theorem}\label{thm:rrt2n-tsn}
$\rca \vdash (\forall n)[\rrt^{n+1}_2 \imp \ts(n)]$
\end{theorem}
\begin{proof}
Fix some~$n \in \Nb$ and let $f : [\Nb]^{n} \to \Nb$ be a coloring. We build a $\Delta^{0,f}_1$
2-bounded coloring $g : [\Nb]^{n+1} \to \Nb$
such that every infinite rainbow for $g$ is, up to finite changes, thin for $f$.
For every $y \in \Nb$ and $\vec{z} \in [\Nb]^n$, if $f(\vec z) = \tuple{x,y}$ with $x < y < min(\vec z)$,
then set $g(y, \vec z) = g(x, \vec z)$. Otherwise assign $g(y, \vec z)$ a fresh color.
The function~$g$ is clearly 2-bounded.
Let $H$ be an infinite rainbow for $g$ and let~$x, y \in H$ be such that~$x < y$.
Set~$H_1 = H \setminus [0, y]$. We claim that~$H_1$ is $f$-thin with color~$\tuple{x,y}$.
Indeed, for every~$\vec z \in [H_1]^n$, if~$f(\vec z) = \tuple{x,y}$
then $x < y < min(\vec z)$, so~$g(x, \vec z) = g(y, \vec z)$.
This contradicts the fact that~$H$ is a~$g$-rainbow.
\end{proof}

\begin{theorem}\label{thm:rrt2n-fsn}
For every standard~$n$, 
$\rca \vdash \rrt^{2n+1}_2 \imp \fs(n)$
\end{theorem}
\begin{proof}
Let~$\tuple{\cdot,\cdot}$ be a bijective coding from~$\{(x, y) \in \Nb^2 : x < y\}$ to~$\Nb$,
such that~$\tuple{x,y} < \tuple{u,v}$ whenever~$x < u$ and~$y < v$. We shall refer to this property as (P1).
We say that a function~$f : [\Nb]^{n} \to \Nb$ is $t$-trapped for some~$t \leq n$
if for every $\vec z \in [\Nb]^n$, $z_{t-1} \leq f(\vec z) < z_t$, 
where~$z_{-1} = -\infty$ and~$z_n = +\infty$.
Wang proved in~\cite[Lemma 4.3]{Wang2014Some} that we can restrict without loss of generality to trapped functions
when $n$ is a standard integer.

Let $f : [\Nb]^{n} \to \Nb$ be a $t$-trapped coloring for some~$t \leq n$. 
We build a $\Delta^{0,f}_1$ 2-bounded coloring $g : [\Nb]^{2n+1} \to \Nb$
such that every infinite rainbow for $g$ computes an infinite set thin for $f$.
Given some~$\vec z \in [\Nb]^n$, we write~$\vec z \bowtie_t u$ to denote
the $(2n+1)$-uple
\[
x_0, y_0, \dots, x_{t-1}, y_{t-1}, u, x_t, y_t, \dots, x_{n-1}, y_{n-1}
\]
where~$z_i = \tuple{x_i, y_1}$ for each~$i < n$.
We say that~$\vec z \bowtie_t u$ is \emph{well-formed} if the sequence above is a strictly increasing.

For every $y \in \Nb$ and $\vec{z} \in [\Nb]^n$ such that~$\vec z \bowtie_t y$ is well-formed, 
if $f(\vec z) = \tuple{x,y}$ for some~$x$ such that~$\vec z \bowtie_t x$ is well-formed, 
then set $g(\vec z \bowtie_t y) = g(\vec z \bowtie_t x)$.
Otherwise assign $g(\vec z \bowtie_t y)$ a fresh color.
The function~$g$ is total and 2-bounded.

Let $H = \{ x_0 < y_0 < x_1 < y_1 < \dots \}$ be an infinite rainbow for $g$ and let
$H_1 = \{ \tuple{x_i, y_i} : i \in \Nb \}$. We claim that~$H_1$ is $f$-free.
Let~$\vec z \in [H_1]^n$ be such that~$f(\vec z) \in H_1$.
In particular, $f(\vec z) = \tuple{x_i, y_i}$ for some~$i \in \Nb$.
By $t$-trapeness of~$f$ and by (P1), if~$f(\vec z) \neq z_{t-1}$
then~$\vec z \bowtie_t x_i$ and~$\vec z \bowtie_t y_i$ are both well-formed.
Hence~$g(\vec z \bowtie_t x_i) = g(\vec z \bowtie_s y_i)$. Because~$H$ is a $g$-rainbow, either~$x_i$
or~$y_i$ is not in~$H$, contradicting~$\tuple{x_i,y_i} \in H_1$. Therefore~$f(\vec z) = z_{t-1}$.
\end{proof}

\begin{corollary}
$\rrt$ and~$\fs$ coincide over~$\omega$-models.
\end{corollary}

We now strengthen Wang's result by proving that $\ts(2) \rightarrow \rrt^2_2$
using Miller's characterization (Theorem~\ref{thm-rrt22-dnrzp}).
We will see in Corollary~\ref{cor:rrt22-sts2} that the implication is strict by showing that $\wwkls{n}$
does not imply $\sts(2)$ over $\rca$ for every~$n$.

\begin{theorem}\label{thm:ts2-dnrzp}
$\rca \vdash \ts(2) \imp \dnrzp$
\end{theorem}
\begin{proof}
We prove that for every set $X$, is an $X$-computable coloring function $f : [\Nb]^2 \to \Nb$
such that every infinite set thin for $f$ computes (not relative to $X$) a function d.n.c. relative to $X^{'}$. 
The structure of the proof is very similar to Theorem~\ref{thm:em-dnrzp}, but instead of diagonalizing against
computing an infinite transitive tournament, we will diagonalize against computing an infinite set avoiding color $i$. 
Applying diagonalization for each color $i$, we will obtain the desired result. 

Let $X$ be a set and $g(.,.)$ be a total $\Delta^{0,X}_1$ function such that $\Phi_e^{X'}(e)=\lim_s g(e,s)$ if the limit exists, and $\Phi^{X'}_e(e) \uparrow$ if the limit does not exist. For each~$e, i, s \in \Nb$,
interpret $g(e,s)$ as the code of a finite set $D_{e,i,s}$ of size $3^{e\cdot i}$. 
We define the coloring~$f$ by $\Sigma_1$-induction as follows. Set $f_0=\emptyset$. At stage~$s+1$, do the following. Start with $f_{s+1}=c_s$. Then, for each $\alpha(e,i) <s$ -- where $\alpha(.,.)$ is the Cantor pairing function, i.e., 
$\alpha(e,i) = \frac{(e+i)(e+i+1)}{2}+e$ -- take the first element $x \in D_{e,i,s} \setminus \bigcup_{(e',i') < (e,i)} D_{e',i',s}$ (notice that these exist by cardinality assumptions on the $D_{e,i,s}$), and if $f_{s+1}(s,x)$ is not already assigned, assign it to color $i$. Finally, for any $z<s$ such that $f_{s+1}(s,z)$ remains undefined, assign any color to it in a predefined way (e.g., for any such pair $\{x,y\}$, set $f_{s+1}(x,y)$ to be $0$). This finishes the construction of $f_{s+1}$. Set $f=\bigcup_s f_s$, which must exist as a set by $\Sigma_1$-induction. 

First of all, notice that $f$ is a coloring function of domain $[\N]^2$, as at the end of stage~$s+1$ of the construction $f(x,y)$ is assigned a value for (at least) all pairs $\{x,y\}$ with $x<s$ and $y<s$. By $\ts(2)$, let~$A$ be an infinite set thin for $f$.
Let $i \in \Nb \setminus f([A]^2)$. Let $h(e)$ be the code of the finite set $A_e$ consisting of the first $3^{e\cdot i+1}$ elements of~$A$. We claim that $h(e) \not=\Phi_e^{X'}(e)$ for all~$e$, which would prove $\dnrzp$. Suppose otherwise, i.e., suppose that $\Phi_e^{X'}(e)=h(e)$ for some~$e$. Then there is a stage $s_0$ such that $h(e)=g(e,s)$ for all~$s \geq s_0$ or equivalently $D_{e,i,s}=A_e$ for all~$s \geq s_0$. Let $N_e=\max(A_e)$. The same argument as in the proof of Theorem~\ref{thm:em-dnrzp} shows that for any $s$ be bigger than both $\max(\bigcup_{e,i,s < N_e} D_{e,i,s})$ and $s_0$, the restriction of $f$ to $A_e \cup \{s\}$ does not avoid color $i$, which contradicts the fact that the infinite set $A$ containing $A_e$ avoids color $i$ in~$f$.
\end{proof}

\begin{corollary}
$\rca \vdash \ts(2) \imp \rrt^2_2$
\end{corollary}
\begin{proof}
Immediate by Theorem~\ref{thm:ts2-dnrzp} and Theorem~\ref{thm-rrt22-dnrzp}.
\end{proof}

\begin{corollary}
The following are true
\begin{itemize}
  \item $\rca \not \vdash \sfs(2) \to \fs(2)$
  \item $\rca \not \vdash \sts(2) \to \ts(2)$
\end{itemize}
\end{corollary}
\begin{proof}
Immediate by Theorem~\ref{thm-rrt22-dnrzp}, Lemma~\ref{lem-srt-sfs},
Corollary~\ref{cor-srt-rrt} and the fact that~$\rca \vdash \fs(2) \imp \ts(2)$ (see Theorem 3.2 in \cite{Cholak2001Free}).
\end{proof}

Notice that the function constructed in the proof of Theorem~\ref{thm:ts2-dnrzp}
is ``stable by blocks'', i.e., for each $x$, there is a 
color $i$ and a finite set $X$ containing $x$ such that 
$(\forall^{\infty} s) i \in f(X, s)$. This can be exploited to prove the implication
from a \emph{polarized} version of Ramsey theorem for pairs.

\begin{definition}[Increasing polarized Ramsey's theorem]
Fix $n, k \geq 1$ and $f : [\Nb]^n \to k$.
\begin{itemize}
	\item An \emph{increasing p-homogeneous} set for $f$ is a sequence $\tuple{H_1, \dots, H_n}$
  of infinite sets such that for some $i < k$, $f(x_1, \dots, x_n) = i$
  for every \emph{increasing} tuple $\tuple{x_1, \dots, x_n} \in H_1 \times \dots \times H_n$.
	\item $\ipt^n_k$ is the statement ``Every coloring $f : [\Nb]^n \to k$
  has an increasing p-homogeneous set.''
\end{itemize}
\end{definition}

Dzhafarov et al.\ proved in~\cite{Dzhafarov2009polarized} 
that $\rca \vdash \ipt^n_k \biimp \aca$ for each $n \geq 3$ and $k \geq 2$.
They also proved that $\ipt^2_2$ lies between $\rt^2_2$ and $\srt^2_2$
and asked which of the implications is strict. We will prove that $\srt^2_2$
does not imply $\ipt^2_2$ over $\rca + \bst$ using the non-standard model of $\srt^2_2$
constructed by Chong et al.\ in~\cite{Chong2014metamathematics}.

\begin{theorem}\label{thm:ipt22-dnrzp}
$\rca \vdash \ipt^2_2 \imp \dnrzp$
\end{theorem}
\begin{proof}
Fix a set~$X$ and let~$T$ be the tournament constructed in the proof of Theorem~\ref{thm:em-dnrzp}.
We can see~$T$ as a function~$f : [\Nb]^2 \to 2$ defined for each~$x < y$ by~$f(x,y) = T(x,y)$.
Let $\tuple{H_0, H_1}$ be an increasing p-homogeneous set for $f$.
Define $h(e)$ to be the code of the finite set $A_e$ consisting
of the first $3^{e+1}$ elements of $H_0$.
We claim that $h(e) \neq \Phi_e^{X'}(e)$ for all $e$, which would prove $\dnrzp$.
Suppose for the sake of contradiction that $\Phi_e^{X'}(e) = h(e)$ for some $e$.
Then there is a stage $s_0$ such that $h(e) = g(e,s)$ for all $s \geq s_0$,
or equivalently $D_{e,s} = A_e$ for all $s \geq s_0$.
Let $N_e = max(A_e)$. We claim that for any $s$ bigger than both $max(\bigcup_{e,s < N_e} D_{e,s})$
and $s_0$, $\{0,1\} \subset f(A_e, s)$. As $A_e \subseteq H_0$, this contradicts the fact
that $\tuple{H_0, H_1}$ is increasing p-homogeneous set.
The proof of the claim is the same as in~Theorem~\ref{thm:em-dnrzp}.
\end{proof}

\begin{corollary}
$\rca + \bst + \srt^2_2 \not \vdash \ipt^2_2$
\end{corollary}
\begin{proof}
Straightforward using Theorem~\ref{thm:ipt22-dnrzp} and Corollary~\ref{cor-srt-rrt}.
\end{proof}

One can generalize Theorem~\ref{thm:ts2-dnrzp} to arbitrary jumps by a simple iteration.

\begin{theorem}[$\rca$]\label{thm:tsk-dnrzk}
For every standard $k \geq 1$, $\rca \vdash \ts(k+1) \imp \dnr[0^{(k)}]$.
\end{theorem}
\begin{proof}
We will prove our theorem by induction over $k \geq 1$
that for every $X \subseteq \Nb$, there is an $X$-computable coloring function $f : [\Nb]^{k+1} \to \Nb$
such that every infinite set thin for $f$ computes (not relative to $X$) a function d.n.c.\ relative to $X^{(k)}$.
Case $k=1$ is exactly the proof of Theorem~\ref{thm:ts2-dnrzp}.

Assume it holds for some $k \in \Nb$. Fix an $X'$-computable
coloring $f : [\Nb]^k \to \Nb$ such that every infinite set thin for $f$ computes
a function d.n.c.\ relative to $(X^{(k-1)})' = X^k$.
Using the Limit Lemma, there exists an $X$-computable approximation function $g: [\Nb]^{k+1} \to \Nb$
such that $\lim_s g(\vec{x}, s) = f(\vec{x})$ for every $\vec{x} \in [\Nb]^k$.
We claim that every infinite set thin for $g$ computes a function d.n.c.\ relative to $X^{(k)}$.

Let $A$ be an infinite set thin for $g$ avoiding some color $i$. If $f(\vec{x}) = i$,
then $g(\vec{x}, s) = i$ for all but finitely many $s$. Hence the set $A$ must be finite. Contradiction.
\end{proof}

\begin{corollary}
$\rca + \rt^2_2 \nvdash \ts(3)$
\end{corollary}
\begin{proof}
By Cholak, Jockusch and Slaman \cite{Cholak2001strength}, there exists an $\omega$-model $\Mcal \models \rca + \rt^2_2$
containing only $\Delta^0_3$ sets. By Theorem~\ref{thm:tsk-dnrzk}, if $\Mcal \models \ts(3)$ 
then $\Mcal \models \dnr[\emptyset'']$ but such a model can't contain only $\Delta^0_3$ sets.
\end{proof}

\begin{theorem}[$\rca + \ist$]\label{thm:sts-no-low}
There exists a computable stable coloring $f : [\Nb]^2 \to \Nb$ with no low infinite set thin for $f$.
\end{theorem}
\begin{proof}
This is a straightforward adaptation of the proof of~\cite{Downey20010_2}.
We assume that definitions and the procedure $P(m)$ has been defined like in the original proof.
Given a stable coloring $f : [\Nb]^2 \to \Nb$, define $A_i = \set{x \in \Nb : (\forall^{\infty} s)f(x, s) \neq i}$.

We need to satisfy the following requirements for all $\Sigma^0_2$ sets $U$, all partial computable
binary functions $\Psi$ and all $i \in \Nb$:
$$
\Rcal_{U, \Psi, i} : U \subseteq A_i \wedge U \in \Delta^0_2 \wedge U \mbox{ infinite } \wedge
\forall n(\lim_s \Psi(n, s) \mbox{ exists}) \imp U' \neq \lim_s \Psi(-,s)
$$

The strategy for satisfying a single requirement $\Rcal_{U, \Psi, i}$ is almost the same.
It begins by choosing an $e \in \Nb$. Whenever a number $x$ enters $U$, 
it enumerates the axiom $\tuple{e, \set{x}}$ for $U'$. Whenever it sees that $\Psi(e, s) \downarrow \neq 1$
for some new number $s$, it \emph{commits} every $x$ for which it has enumerated an axiom $\tuple{e, \set{x}}$
to be assigned color $i$, i.e. starts settings $f(x, t) = i$ for every $t \geq s$.

If $U$ is $\Delta^0_2$ and infinite, and $\lim_s \Psi(e, s)$ exists and is not equal to 1,
then eventually an axiom $\tuple{e, \set{x}}$ for some $x \in U$ is enumerated, in which case
$U'(e) = 1 \neq \lim_s \Psi(e,s)$. On the other hand, if $U \subseteq A_i$ and $\lim_s \Psi(e,s) = 1$
then for all axioms $\tuple{e, \set{x}}$ that are enumerated by our strategy, $x$ is eventually commited to 
be assigned color $i$, which implies that $x \not \in U$. Thus in this case, $U'(e) = 0 \neq \lim_s \Psi(e,s)$.

The global construction is exactly the same as in the original proof.
\end{proof}

\begin{theorem}
$\rca \vdash \emo \rightarrow [\sts(2) \vee \coh]$
\end{theorem}
\begin{proof}
Let $f : [\Nb]^2 \to \Nb$ be a stable coloring and $R_0, R_1, \dots$ be a uniform sequence of sets.
We denote by $\tilde{f}$ the function defined by $\tilde{f}(x) = \lim_s f(x,s)$.
We build a $\Delta^{0, f \oplus \vec{R}}_1$ tournament $T$ such that every infinite transitive subtournament
is either thin for $\tilde{f}$ or is an $\vec{R}$-cohesive set. As every set $H$ thin for $\tilde{f}$
$H \oplus f$-computes a set thin for $f$, we are done.
For each $x, s \in \Nb$, set $T(x,s)$ to hold if one of the following holds:
\begin{itemize}
	\item[(i)] $f(x,s) = 2i$ and $x \in R_i$  
	\item[(ii)] $f(x,s) = 2i+1$ and $x \not \in R_i$
\end{itemize}
Otherwise set $T(s,x)$ to hold. Let $H$ be an infinite transitive subtournament of $T$
which is not $\tilde{f}$-thin. Suppose for the sake of absurd that $H$ is not $\vec{R}$-cohesive.
Then there exists an $i \in \Nb$ such that $H$ intersects $R_i$ and $\overline{R_i}$ infinitely many times.
As $H$ is not $\tilde{f}$-thin, there exists $x, y \in H$ such that $\tilde{f}(x) = \lim_s f(x,s) = 2i$
and $\tilde{f}(y) = \lim_s f(y,s) = 2i+1$. As $H$ intersects $R_i$ and $\overline{R_i}$ infinitely many times,
there exists $s_0 \in R_i \cap H$ and $s_1 \in \overline{R_i} \cap H$ 
such that $f(x,s_0) = f(x,s_1) = 2i$ and $f(y,s_0) = f(y,s_1) = 2i+1$.
But then $T(x,s_0)$, $T(s_0, y)$, $T(y, s_1)$ and $T(s_1, x)$ hold, forming a 4-cycle
and therefore contradicting transitivity of $H$.
\end{proof}

\begin{definition}[Atomic model theorem]
A formula $\varphi(x_1, \dots, x_n)$ of $T$ is an \emph{atom} of a theory $T$ if for each formula $\psi(x_1, \dots, x_n)$
  we have $T \vdash \varphi \imp \psi$ or $T \vdash \varphi \imp \neg \psi$ but not both.
  A theory $T$ is \emph{atomic} if, for every formula $\psi(x_1, \dots, x_n)$ consistent with $T$,
  there is an atom $\varphi(x_1, \dots, x_n)$ of $T$ extending it, i.e. one such that $T \vdash \varphi \imp \psi$.
  A model $\Acal$ of $T$ is \emph{atomic} if every $n$-tuple from $\Acal$ satisfies an atom of $T$. 
$\amt$ is the statement ``Every complete atomic theory has an atomic model''.
\end{definition}

$\amt$ has been introduced as a principle by Hirschfeldt et al.\ in~\cite{Hirschfeldt2009atomic} together with $\pizog$.
They also proved that $\wkl$ and $\amt$ were incomparable on $\omega$-models, 
proved over $\rca$ that $\amt$ is strictly weaker than $\sads$.
They proved that $\amt$ is restricted $(r\h)\Pi^1_2$ conservative over $\rca$, deduced from it that $\amt$ implied none
of $\rt^2_2$, $\srt^2_2$, $\cac$, $\cads$ or even $\dnr$. 
They finally proved that $\amt$ and $\pizog$ have the same $\omega$-models.
Bienvenu et al.\ proved in~\cite{Bienvenu2014role} the existence of an $\omega$-model of $\rrt^2_2$ 
which is not a model of $\amt$.

\begin{theorem}\label{thm:sts2-amt}
$\rca \vdash \sts(2) \imp \amt$
\end{theorem}
\begin{proof}
We prove that for any atomic theory $T$, there exists a $\Delta^{0,T}_1$ stable
coloring $f : [\Nb]^2 \to \Nb$ such that for any infinite
set $H$ thin for $f$, there is a~$\Delta^{0, H \oplus T}_1$ atomic model of $T$.
The proof is very similar to~\cite[Theorem~4.1]{Hirschfeldt2009atomic}.
We begin again with an atomic theory $T$ and consider the tree $\Scal$ of standard
Henkin constructions of models of $T$. We want to define a stable coloring $f : [\Nb]^2 \to \Nb$
such that any infinite set thin for $f$ computes an infinite path $P$ through $\Scal$
that corresponds to an atomic model $\Acal$ of $T$.
Define as in~\cite[Theorem~4.1]{Hirschfeldt2009atomic} 
a monotonic computable procedure $\Phi$ which on tuple $\tuple{x_1, \dots, x_n}$
will return a tuple $\tuple{\sigma_1, \dots, \sigma_n}$ such that $\sigma_{i+1}$ is the
least node of $\Scal$ extending $\sigma_i$ such that we have found no witness that $\sigma_{i+1}$
is not an atom about $c_0, \dots, c_{x_i}$ after a standardized search of $x_{i+1}$ many steps.
$\sigma_1$ is the least node on $\Scal$ mentionning $c_0$ and such that we have not found a witness
that $\sigma_1$ is not an atom about $c_0$ after $x_1$ many steps.

The construction of the coloring $f$ will involve a movable marker procedure.
At each stage $s$, we will ensure to have defined $f$ on $\set{x : x \leq s}$.
For each color $i$, we can associate the set $C_i = \{x : (\forall^{\infty} s)f(x,s) = i\}$.
At stage $s$, we maintain a set $C_{i,s}$ with the intuition that $C_i = \lim_s C_{i,s}$.

For each~$e, i \in \Nb$, the requirement
$\Rcal_{e,i}$ ensures that for any sequence
$x_1, \dots, x_n, d_{e,i,t}$ in $\overline{C_i}$ that is increasing
in natural order, $\sigma_{n+1}$ includes an atom about $c_0, \dots, c_{x_n}$
where $d_{e,i,t}$ is the value of the marker~$d_{e,i}$ associated to~$\Rcal_{e,i}$
at stage~$t$, and $\Phi(x_1, \dots, x_n, d_{e,i,t}) = \tuple{\sigma_1, \dots, \sigma_{n+1}}$.

We say that the requirement $\Rcal_{e,i}$ \emph{needs attention at stage $s$}
if there exists a sequence $x_1, \dots, x_n, d_{e,i,s}$ of elements of $\overline{C_{i,s}}$ increasing in natural order,
such that $\Phi(x_1, \dots, x_n, d_{e,i,s}) = \tuple{\sigma_1, \dots, \sigma_{n+1}}$ and by stage $s$ we have seen
a witness that $\sigma_{n+1}$ does not supply an atom about $c_0, \dots, c_{x_n}$.

At stage $s$, suppose the highest priority requirement needing attention is $\Rcal_{e,i}$.
The strategy \emph{commits to $C_i$} each $x < s$ that are in greater or equal to $d_{e,i,s}$.
We let $d_{e, i,s+1} = s$. All $d_{u, j, s+1}$ become undefined for $\tuple{u,j} > \tuple{e,i}$.
If no requirement needs attention, we let $d_{u,j,s+1} = s$
for the least $\tuple{u,j}$ such that $d_{u,j,s}$ is undefined.
For each $x < s$, set $f(x, s) = i$ if $x$ is commited to be in $C_i$.
Otherwise set $f(x, s) = 0$. We then go to the next stage.

\begin{claim}
The resulting coloring is stable.
\end{claim}
\begin{proof}
Take any $x \in \Nb$. If no requirement ever commits $x$ to be in some $D_i$
then $x$ is commited at stage $x+1$ to be in $C_0$ and this commitment is never injured,
so $(\forall^\infty s)f(x,s) = 0$. Otherwise by $\isig^0_1$ there is a requirement
of highest priority that commits $x$ to be in some $C_i$.
Say it is $\Rcal_{e,i}$ and it acts to commit $x$ at stage $s$.
This means that $d_{e,i,s} \leq x < s$. Then we set $d_{e,i,s+1} = s$ and never decrease this marker.
No requirement of higher priority will act after stage $s$ on $x$ by our choice of $\Rcal_{e,i}$
and the markers for strategies of lower priority will be initialized after stage $s$ to 
a value greater than $s$. So $x$ will stay for ever in $C_i$. Thus $(\forall^\infty s)f(x,s) = i$.
\end{proof}

\begin{claim}
Each requirement $\Rcal_{e,i}$ acts finitely often and
$d_{e,i,s}$ will eventually remain fixed. Moreover,
if $d_{e,i,s}$ never changes after stage $t$, then, for any sequence
$x_1, \dots, x_n, d_{e,i,t}$ in $\overline{C_i}$ that is increasing
in natural order, $\sigma_{n+1}$ includes an atom about $c_0, \dots, c_{x_n}$
where $\Phi(x_1, \dots, x_n, d_{e,i,t}) = \tuple{\sigma_1, \dots, \sigma_{n+1}}$.
\end{claim}
\begin{proof}
We prove it by $\Sigma^0_1$ induction.
Assume that $\Rcal_{e,i}$ acts at stage $s$ and no requirement of higher
priority ever acts again. We then set $d_{e,i,s+1} = s$ and 
now act again for $\Rcal_{e,i}$ only if we discover a new witness as described
in the definition of needing attention. As we never act for any requirement
of higher priority, at any stage $t > s$ the numbers between $d_{e,i,s}$ and $d_{e,i,t}$
will all be commited to~$C_i$. Then the sequences $x_1, \dots, x_n \leq d_{e,i,t}$ in $\overline{C_i}$,
increasing in natural order are sequences $x_1, \dots, x_n \leq d_{e,i,s}$ in $\overline{C_i}$.
Hence their set is bounded. By the same trick as in~\cite[Theorem~4.1]{Hirschfeldt2009atomic},
we can avoid the use of $\bst$ by constructing a single atom extending each $\sigma(x_1, \dots, x_n)$
where $\sigma(x_1, \dots, x_n)$ is the next to last value under $\Phi$.
By $\ist$, there is a first such atom and a bound on the witnesses needed to show that all smaller candidates
are not such an atom. Once we passed such a stage, no change occurs in $d_{e,i,t}$ and
its value must also be above the stage where all witnesses are found. After such
a stage, $\Rcal_{e,i}$ will never need attention again.
\end{proof}

The construction of an atomic model of $T$ from any infinite set thin for $f$
with witness color $i$ is exactly the same as in~\cite[Theorem~4.1]{Hirschfeldt2009atomic}.
\end{proof}

\begin{corollary}\label{cor:wwkl-sts2}
For every~$n$, $\wwkls{n}$ does not imply~$\sts(2)$.
\end{corollary}
\begin{proof}
Bienvenu et al.~\cite{Bienvenu2014role} have shown the existence of a computable complete atomic theory~$T$
such that the measure of oracles computing an atomic model of~$T$ is null.
Therefore there exists an $\omega$-model of $\wwkls{n}$ which is not a model of~$\amt$,
and a fortiori which is not a model of~$\sts(2)$.
\end{proof}

\begin{corollary}\label{cor:rrt22-sts2}
$\rrt^2_2$ does not imply~$\sts(2)$ over~$\rca$.
\end{corollary}
\begin{proof}
Csima \& Mileti~\cite{Csima2009strength} proved that~$\rca \vdash \wwkls{2} \imp \rrt^2_2$.
Apply Corollary~\ref{cor:wwkl-sts2}.
\end{proof}

\begin{question}
Does $\fs(2)$ imply $\bst$ over $\rca$ ?
\end{question}

\begin{question}
Does $\fs(2)$ imply $\semo$ over $\rca$ ?
\end{question}

\section{Stable rainbow Ramsey theorem}

In this section, we study a stable version of the rainbow Ramsey theorem.
There exist different notions of stability for $k$-bounded functions.
The naturality of this version is justified by the existence of various
simple characterizations of the stable rainbow Ramsey theorem for pairs.
We shall later study another version which seems more natural in the sense 
that a stable instance can be obtained from a non-stable one by an application
of the cohesiveness principle. However the latter version does not admit
immediate simple characterizations.

\subsection{Definitions}

\begin{definition}
A 2-bounded coloring $f : [\Nb]^2 \to \Nb$ is 
\emph{strongly rainbow-stable} if  
$(\forall x)(\exists y \neq x)(\forall^\infty s)f(x, s) = f(y,s)$
A set $X \subseteq \Nb$ is a \emph{prerainbow} for a 2-bounded coloring $f : [\Nb]^2 \to \Nb$ if 
$(\forall x \in X)(\forall y \in X)(\forall^\infty s \in X)[f(x, s) \neq f(y, s)]$.
$\srrt^2_2$ is the statement ``every rainbow-stable 2-bounded coloring $f : [\Nb]^2 \to \Nb$
has a rainbow.''
\end{definition}

\begin{lemma}[Wang in \cite{WangSome}, $\rca + \bst$]\label{lem:prerainbow-equiv}
Let $f : [\Nb]^2 \to \Nb$ be a $2$-bounded coloring and $X$ be an infinite prerainbow for $f$. 
Then $X \oplus f$ computes an infinite $f$-rainbow $Y \subseteq X$.
\end{lemma}

\begin{theorem}\label{thm:rainbow-stable-vs-strongly}
The following are equivalent over $\rca + \bst$:
\begin{itemize}
  \item[(i)] $\srrt^2_2$
  \item[(ii)] Every strongly rainbow-stable 2-bounded coloring $f : [\Nb]^2 \to \Nb$ has a rainbow.
\end{itemize}
\end{theorem}
\begin{proof}
$(i) \imp (ii)$ is straightforward as any strongly rainbow-stable coloring is rainbow-stable.
$(ii) \imp (i)$: Let $f : [\Nb]^2 \to \Nb$ be a 2-bounded rainbow-stable coloring.
Consider the following collection:
$$
S = \set{ x \in \Nb : (\forall^{\infty} s)(\forall y \neq x)[f(y, s) \neq f(x,s)]}
$$

If $S$ is finite, then take $n \geq max(S)$. The restriction of $f$ to $[n, +\infty)$
is a strongly rainbow-stable 2-bounded coloring and we are done. So suppose $S$ is infinite.
We build a 2-bounded strongly rainbow-stable  coloring $g \leq_T f$ by stages.

At stage $t$, assume $g(x, i)$ is defined for every $x, i < t$.
For every pair $x, y \leq t$ such that $f(x, t) = f(y, t)$,
define $g(x, t) = g(y, t)$. Let $S_t$ be the set of $x \leq t$
such that $g(x, t)$ has not been defined yet.
Writing $S_t = \set{x_1 < x_2 < \dots}$, we set $g(x_{2i}) = g(x_{2i+1})$ for each $i$.
If $S_t$ has an odd number of elements, there remains an undefined value.
Set it to a fresh color.
This finishes the construction. 
It is clear by construction that $g$ is 2-bounded.

\begin{claim}
$g$ is strongly rainbow-stable.
\end{claim}
\begin{proof}
Fix any $x \in \Nb$. Because $f$ is rainbow-stable, we have two cases:
\begin{itemize}
  \item Case 1: there is a $y \neq x$ such that $(\forall^{\infty} s) f(x, s) = f(y, s)$.
  Let $s_0$ be the threshold such that $(\forall s \geq s_0) f(x, s) = f(y, s)$.
  Then by construction, at any stage $s \geq s_0$, $g(x, s) = g(y, s)$ and we are done.
  
  \item Case 2: $x \in S$. Because $x$ is infinite, it has a successor $y_0 \in S$. 
  By $\bst$, let $s_0$ be the threshold such that for every $y \leq y_0$
  either there is a $z \leq y_0$, $z \neq y$ such that $(\forall s \geq s_0) f(y, s) = f(z, s)$ 
  or $(\forall s \geq s_0)$ $f(y, s)$ is a fresh color.
  Then by construction of $g$, for every $t \geq s_0$, $S_t \uh y = S \uh y$.
  Either $x = x_{2i}$ for some $i$ and then $(\forall t \geq s_0) g(x, t) = g(x_{2i+1}, t)$
  or $x = x_{2i+1}$ for some $i$ and then $(\forall t \geq s_0) g(x, t) = g(x_{2i}, t)$.
\end{itemize}
\end{proof}

\begin{claim}
Every infinite prerainbow for $g$ is a prerainbow for $f$.
\end{claim}
\begin{proof}
Let $X$ be an infinite prerainbow for $g$ and assume for the sake of contradiction
that it is not a prerainbow for $f$. Then there exists two elements $x, y \in X$
such that $(\forall s)(\exists t \geq s)[f(x, t) = f(y, t)]$.
But then because $f$ is rainbow-stable, 
there is a threshold $s_0$ such that $(\forall s \geq s_0)[f(x, s) = f(y, s)]$.

Then by construction of $g$, for every $s \geq s_0$, $g(x, s) = g(y, s)$.
For every $u \in X$ there is an $s \in X$ with $s \geq u, s_0$ such that
$g(x, s) = g(y, s)$ contradicting the fact that $X$ is a prerainbow for $g$.
\end{proof}

Using Lemma~\ref{lem:prerainbow-equiv}, for any infinite $H$ prerainbow for $g$,
$f \oplus H$ computes an infinite rainbow for $f$. This finishes the proof.
\end{proof}

\bigskip

\subsection{Relation with diagonal non-recursiveness}\label{sect:srrt22-dnr}

It is well-known that being able to compute a d.n.c.\ function is equivalent
to being able to uniformly find a member outside a finite $\Sigma^0_1$ set if we know an upper bound on its size,
and also equivalent to diagonalize against a $\Sigma^0_1$ function. The proof relativizes well 
and is elementary enough to be formalized in $\rca$ (see Theorem~\ref{thm:dnrzn-char-dnr}).

\begin{definition}
Let $(X_e)_{e \in \Nb}$ be a uniform family of finite sets.
An \emph{$(X_e)_{e \in \Nb}$-escaping function} is a function $f : \Nb^2 \to \Nb$
such that $(\forall e)(\forall n)[\card{X_e} \leq n \imp f(e,n) \not \in X_e]$.
Let $h : \Nb \to \Nb$ be a function. An \emph{$h$-diagonalizing function} $f$ is a function $\Nb \to \Nb$
such that $(\forall x)[f(x) \neq h(x)]$.
When $(X_e)_{e \in \Nb}$ and $h$ are clear from context, they may be omitted.
\end{definition}

\begin{theorem}[Folklore]\label{thm:dnrzn-char-dnr}
For every $n \in \Nb$, the following are equivalent over $\rca$:
\begin{itemize}
  \item[(i)] $\dnr[0^{(n)}]$
  \item[(ii)] Any uniform family $(X_e)_{e \in \Nb}$ of $\Sigma^0_{n+1}$ finite sets
  has an escaping function.
  \item[(iii)] Any partial $\Delta^0_{n+1}$ function has a diagonalizing function.
\end{itemize}
\end{theorem}
\begin{proof}
Fix a set $A$.
\begin{itemize}
  \item $(i) \imp (ii)$:
  Let $(X_e)_{e \in \Nb}$ be a uniform family of finite $\Sigma^{0,A}_{n+1}$ finite sets
  and $f$ be a function d.n.c.\ relative to $A^{(n)}$. Define a function $h : \Nb^2 \to \Nb$
  by $h(e, n) = \tuple{f(i_1), \dots, f(i_n)}$ where $i_j$ is the index of the
  partial $\Delta^{0, A}_{n+1}$ function which on every input, looks at the $j$th element $k$ of $X_e$
  if it exists, interprets $k$ as a $n$-tuple $\tuple{k_1, \dots, k_n}$ and returns $k_j$.
  The function diverges if no such $k$ exists.
  One easily checks that $h$ is an $(X_e)_{e \in \Nb}$-escaping function.
  
  \item $(ii) \imp (iii)$:
  Let $f : \Nb \to \Nb$ be a partial $\Delta^{0,A}_{n+1}$ function.
  Consider the enumeration defined by $X_e = \set{f(e)}$ if it $f(e) \downarrow$ and $X_e = \emptyset$ otherwise.
  This is a uniform family of $\Sigma^{0,A}_{n+1}$ finite sets, each of size at most 1. Let $g: \Nb^2 \to \Nb$ be an
  $(X_e)_{e \in \Nb}$-escaping function. Then $h : \Nb \to \Nb$ defined by $h(e) = g(e, 1)$ is
  an $f$-diagonalizing function.  
  
  \item $(iii) \imp (i)$:
  Consider the partial $\Delta^{0,A}_{n+1}$ function $f(e) = \Phi^{A^{(n)}}_e(e)$. 
  Any $f$-diagonalizing function is d.n.c relative to $A^{(n)}$.
\end{itemize}
\end{proof}

In particular, using Miller's characterization of $\rrt^2_2$ by $\dnrzp$, we have the following theorem taking $n=1$:

\begin{theorem}[Folklore]\label{thm:rrt22-char-dnr}
The following are equivalent over $\rca$:
\begin{itemize}
  \item[(i)] $\rrt^2_2$
  \item[(ii)] Any uniform family $(X_e)_{e \in \Nb}$ of $\Sigma^0_2$ finite sets
  has an escaping function.
  \item[(iii)] Any partial $\Delta^0_2$ function has a diagonalizing function.
\end{itemize}
\end{theorem}

In the rest of this section, we will give an equivalent of Theorem~\ref{thm:rrt22-char-dnr}
for $\srrt^2_2$.

\begin{lemma}[$\rca + \bst$]\label{lem:srrt22-char1}
For every $\Delta^0_2$ function $h : \Nb \to \Nb$,
there exists a computable rainbow-stable 2-bounded coloring $c : [\N]^2 \to \Nb$
such that every infinite rainbow $R$ for $c$ computes an $h$-diagonalizing function.
\end{lemma}
\begin{proof}
Fix a $\Delta^0_2$ function $h$ and a uniform family $(D_e)_{e \in \Nb}$ of all finite sets.
We will construct a rainbow-stable 2-bounded coloring $c : [\Nb]^2 \to \Nb$ by
a finite injury priority argument. By Schoenfield's limit lemma, there exists
a total computable function $g(\cdot, \cdot)$ such that $\lim_s g(x, s) = h(x)$ for every $x$.

Our requirements are the following:

\bigskip
$\Rcal_x$: If $\card{D_{\lim_s g(x, s)}} \geq 3x+2$ then $\exists u, v \in D_{\lim_s g(x, s)}$
such that $(\forall^{\infty} s) c(v, s) = c(v, s)$.
\smallskip

We first check that if every requirement is satisfied then we can compute
a function $f : \Nb \to \Nb$ such that $(\forall x)[f(x) \neq h(x)]$ from
any infinite rainbow for $c$. Fix any infinite set $R$ rainbow for $c$.
Let $f$ be the function which given $x$ returns the index of the set
of the first $3x+2$ elements of $R$. Because of the requirement $\Rcal_x$,
$D_{f(x)} \neq D_{\lim_s g(x,s)}$. Otherwise $\card{D_{f(x)}} = 3x+2$ and
there would be two elements $u, v \in D_{f(x)} \subset R$
such that $(\forall^{\infty} s) c(x, s) = c(y, s)$. So take an element $s \in R$ large enough to witness this fact.
$c(x, s) = c(y, s)$ for $x, y, s \in R$ contradicting the fact that $R$ is a rainbow.
So $D_{f(x)} \neq D_{\lim_s g(x,s)}$ from which we deduce $f(x) \neq \lim_s g(x, s) = h(x)$. 

Our strategy for satisfying a local requirement $\Rcal_x$ is as follows. 
If $\Rcal_x$ receives attention at stage $t$, it checks whether $\card{D_{g(x, t)}} \geq 3x+2$.
If this is not the case, then it is declared satisfied. If $\card{D_{g(x, t)}} \geq 3x+2$,
then it chooses the least two elements $u, v \geq x$, such that $u, v \in D_{g(x, s)}$ 
and $u$ and $v$ are \emph{not} restrained
by a strategy of higher priority and \emph{commits} to assigning a common color.
For any such pair $u, v$, this commitment will remain active as long as the strategy has a restraint
on that element. Having done all this, the local strategy is declared to be satisfied and will not act again
unless either a higher priority puts restraint on $u$ or $v$ or at a further stage $t' > t$, $g(x, t') \neq g(x, t)$.
In both cases, the strategy gets \emph{injured} and has to reset, releasing all its restraints.

To finish stage $t$, the global strategy assigns $c(u, t)$ for all $u \leq t$ as follows:
if $u$ is commited to some assignment of $c(u, t)$ due to a local strategy, define $c(u, t)$ to be this value.
If not, let $c(u, t)$ be a fresh color. This finishes the construction and we now turn to the verification.
It is easy to check that each requirement restrains at most two elements at a given stage.

\begin{claim}
Every given strategy acts finitely often.
\end{claim}
\begin{proof}
Fix some~$x \in \Nb$.
By~$\bst$ and because $g$ is limit-computable, there exists a stage $s_0$
such that $g(y, s) = g(y, s_0)$ for every~$y \leq x$ and~$s \geq s_0$.
If $|D_{g(x,s_0)}| < 3x+2$, then the requirement is satisfied and does not act any more.
If $|D_{g(x,s_0)}| \geq 3x+2$, then by a cardinality argument, there exists two elements~$u$ and~$v \in D_{g(x,s_0)}$
which are not restrained by a strategy of higher priority. Because $D_{g(y, s)} = D_{g(y, s_0)}$ for each~$y \leq x$
and~$s \geq s_0$, no strategy of higher priority will change its restrains and will therefore injure~$\Rcal_x$ after stage~$s_0$.
So $(\forall^{\infty} s) c(u, s) = c(v,s)$ for some $u, v \in D_{\lim_s g(x, s)}$ and requirement $\Rcal_x$ is satisfied.
\end{proof}

\begin{claim}
The resulting coloring $c$ is rainbow-stable.
\end{claim}
\begin{proof}
Consider a given element $u \in \Nb$. We distinguish three cases:
\begin{itemize}
  \item Case 1: the element becomes, during the construction, free from any restraint after some stage
$t \geq t_0$ . In this case, by construction, $c(u, t)$ is assigned a fresh color for all $t \geq t_0$.
Then $(\forall^{\infty} s)(\forall v \neq u)[c(u, s) \neq c(v, s)]$.

  \item Case 2: there is a stage $t_0$ at which some restraint is put on $u$ by some local strategy, and this
restraint is never released. In this case, the restraint comes together with a commitment that
all values of $c(u, s)$ and $c(v, s)$ be the same beyond some stage $t_0$ for some fixed $v \neq x$. 
Therefore for all but finitely many stages $s$, $c(u, s) = c(v, s)$.

  \item Case 3: during the construction, infinitely many restraints are put on $u$ and are later released. This
  is actually an impossible case, since by construction only strategies for requirements $\Rcal_y$
  with $y \leq u$ can ever put a restraint on $u$. By~$\bst$,
	there exists some stage after which no stragegy ~$\Rcal_y$ acts for every~$y \leq u$
	and therefore the restraints on~$u$ never change again.
\end{itemize}
\end{proof}
This last claim finishes the proof.
\end{proof}

\begin{lemma}[$\rca + \ist$] \label{lem:char1-srrt22}
For every computable strongly rainbow-stable 2-bounded coloring $f : [\Nb]^2 \to \Nb$
there exists a uniform family $(X_e)_{e \in \Nb}$ of $\Delta^0_2$ finite sets
whose sizes are uniformly $\Delta^0_2$ computable
such that every $(X_e)_{e \in \Nb}$-escaping function
computes a rainbow for~$c$.
\end{lemma}
\begin{proof}
Fix any uniform family $(D_e)_{e \in \Nb}$ of finite sets.
Let $f : [\Nb]^2 \to \Nb$ be a 2-bounded rainbow-stable computable coloring.
For an element $x$, define 
$$
\bad(x) = \set{ y \in \Nb : (\forall^\infty s)f(x,s) = f(y,s)}
$$
Notice that $x \in \bad(x)$.
Because $f$ is strongly rainbow-stable, $\bad$ is a $\Delta^0_2$ function.
For a set $S$, $\bad(S) = \bigcup_{x \in S} \bad(x)$.
Define $X_e = \bad(D_e)$. Hence $X_e$ is a $\Delta^0_2$ set, and this uniformly in $e$.
Moreover, $\card{X_e} \leq 2\card{D_e}$ and for every $x$, $\card{\bad(x)} = 2$
so we can $\emptyset'$-compute the size of $X_e$ with the following equality
$$
  \card{X_e} = 2|D_e| - 2 \card{\set{\set{x, y} \subset D_e : \bad(x) = \bad(y)}}
$$

Let $h : \Nb \to \Nb$ be a function satisfying $(\forall e)(\forall n)[\card{X_e} \leq n \imp h(e, n) \not \in X_e]$.
We can define $g : \Nb \to \Nb$ by $g(e) = h(e, 2\card{D_e})$.
Hence $(\forall e)g(e) \not \in X_e$.

We construct a prerainbow $R$ by stages $R_0 (=\emptyset) \subsetneq R_1 \subsetneq R_2, \dots$
Assume that at stage $s$, 
$(\forall \{x,y\} \subseteq R_s)(\forall^{\infty} s)[f(x, s)\neq f(y,s)]$.
Because $R_s$ is finite, we can computably find some index $e$ such that $R_s = D_e$.
Set $R_{s+1} = R_s \cup \set{g(e)}$. By definition, $g(e) \not \in X_e$.
Let $x \in R_s$. Because $g(e) \not \in X_e$, $(\forall^{\infty} s) f(x, s) \neq f(g(e), s)$.
By~$\ist$, the set~$R$ is a prerainbow for~$f$.
By Lemma~\ref{lem:prerainbow-equiv} we can compute an infinite rainbow for $f$ 
from $R \oplus f$.
\end{proof}

\begin{theorem}\label{thm:srrt22-characterizations}
The following are equivalent over $\rca + \ist$:
\begin{itemize}
  \item[(i)] $\srrt^2_2$
  \item[(ii)] Any uniform family $(X_e)_{e \in \Nb}$ of $\Sigma^0_2$ finite sets
  whose sizes are uniformly $\Delta^0_2$ has an escaping function.
  \item[(iii)] Any $\Delta^0_2$ function $h : \Nb \to \Nb$ has a diagonalizing function.
\end{itemize}
\end{theorem}
\begin{proof}
$(i) \imp (iii)$ is Lemma~\ref{lem:srrt22-char1} and $(ii) \imp (i)$ follows from Lemma~\ref{lem:char1-srrt22}.
This is where we use~$\ist$.
We now prove $(iii) \imp (ii)$.
Let $(X_e)_{e \in \Nb}$ be a uniform family of $\Sigma^0_2$ finite sets
such that~$|X_e|$ is~$\Delta^0_2$ uniformly in~$e$. For each~$n, i \in \Nb$,
define~$(n)_i$ to be the $i$th component of the tuple whose code is~$n$, if it exists. Define
\[
h(\tuple{e,i}) = \cond{
	(n)_i & \mbox{ where } n \mbox{ is the } i\mbox{th element of } X_e \mbox{ if } i < |X_e|\\
	0 & \mbox{ oherwise}
}
\]
By (iii), let $g : \Nb \to \Nb$ be a total function such that $(\forall e)[g(e) \neq h(e)]$.
Hence for every pair $\tuple{e,i}$ such that~$i \leq |X_e|$, $g(\tuple{e,i}) \neq (n)_i$
where $n$ is the $i$th element of $X_e$.
Define $f : \Nb^2 \to \Nb$ to return on inputs $e$ and $s$ the tuple $\tuple{g(\tuple{e,0}), \dots, g(\tuple{e,s})}$.
Hence if $s \geq \card{X_e}$ then $f(e, s) \neq m$
where $m$ is the $i$th element of $X_e$ for each~$i < |X_e|$. So $f(e,n) \not \in X_e$.
\end{proof}

\begin{corollary}
Every $\omega$-model of $\srrt^2_2$ is a model of $\dnr$.
\end{corollary}
\begin{proof}
Let $h : \Nb \to \Nb$ be the $\Delta^0_2$ function which on input $e$ returns $\Phi_e(e)$ if $\Phi_e(e) \downarrow$
and returns 0 otherwise. By (iii) of Theorem~\ref{thm:srrt22-characterizations} there
exists a total function $f : \Nb \to \Nb$ such that $(\forall e)[f(e) \neq h(e)]$.
Hence $(\forall e)[f(e)\neq \Phi_e(e)]$ so $f$ is a d.n.c.\ function. 
\end{proof}

\begin{theorem}
$\rca \vdash \srrt^2_2 \imp \dnr$
\end{theorem}
\begin{proof}
If $\Phi_{e}(e) \downarrow$ then interpret $\Phi_e(e)$ as the code of a finite set $D_e$ of size $3^{e+1}$
with $min(D_e) > e$.
Let $D_{e,s}$ be the approximation of $D_e$ at stage $s$, i.e. $D_{e, s}$ is the set $\{e+1, \dots, e+3^{e+1}\}$
if $\Phi_{e,s}(e) \uparrow$ and $D_{e,s} = D_e$ if $\Phi_{e,s}(e) \downarrow$.
We will construct a rainbow-stable coloring $f : [\Nb]^2 \to \Nb$ meeting the following requirements for each $e \in \Nb$.
$$
\Rcal_e : \Phi_e(e) \downarrow \imp (\exists a, b \in D_e)(\forall^{\infty} s)f(a, s) = f(b, s)
$$

Before giving the construction, let us explain how to compute a d.n.c.\ function from any infinite rainbow for $f$
if each requirement is satisfied. Let $H$ be an infinite rainbow for $f$. Define the function
$g : \Nb \to \Nb$ which given $e$ returns the code of the $3^{e+1}$ first elements of $H$.
We claim that $g$ is a d.n.c.\ function. Otherwise suppose $g(e) = \Phi_e(e)$ for some $e$.
Then $D_e \subseteq H$, but by $\Rcal_e$, $(\exists a, b \in D_e)(\forall^{\infty} s)f(a, s) = f(b, s)$.
As $H$ is infinite, there exists an $s \in H$ such that $f(a, s) = f(b, s)$, contradicting the 
fact that $H$ is a rainbow for~$f$.

We now describe the construction. The coloring $f$ is defined by stages.
Suppose that at stage $s$, $f(u,v)$ is defined for each $u, v < s$.
For each $e < s$ take the first pair $\set{a, b} \in D_{e,s} \setminus \bigcup_{k < e} D_{k,s}$.
Such a pair must exist by cardinality assumption on the $D_{e,s}$. Set $f(a, s) = f(b, s) = i$
for some fresh color $i$. Having done that, for any $u$ not yet assigned, assign $f(u, s)$ a fresh color
and go to stage $s+1$.

\begin{claim}
Each requirement $\Rcal_e$ is satisfied.
\end{claim}
\begin{proof}
Fix an $e \in \Nb$.
By $\bsig^0_1$ there exists a stage $s$ such that $\Phi_{k,s}(k) = \Phi_k(k)$ for each $k \leq e$.
Then at each further stage $t$, the same par $\set{a, b}$ will be chosen in $D_{e,s}$
to set $f(a, t) = f(b, t)$. Hence if $\Phi_e(e) \downarrow$, there are $a, b \in D_e$ such that
$(\forall^{\infty} s)f(a, s) = f(b, s)$.
\end{proof}

\begin{claim}
The coloring $f$ is rainbow-stable.
\end{claim}
\begin{proof}
Fix an element $u \in \Nb$. By $\bsig^0_1$ there is a stage $s$ such that 
$\Phi_{k,s}(k) = \Phi_k(k)$ for each $k < u$. If $u \in \set{a, b}$
for some pair $\set{a,b}$ chosen by a requirement of priority $k < u$ then
at any further stage $t$, $f(u, t) = f(a, t) = f(b,t)$. If $u$ is not chosen by any requirement
of priority $k < u$, then $u$ will not be chosen by any further requirement as $min(D_e) > e$ for each $e \in \Nb$.
So by construction, $f(u,t)$ will be given a fresh color for each $t > s$.
\end{proof}
\end{proof}

\subsection{K\"onig's lemma and relativized Schnorr tests}\label{srrt22-konig}

D.n.c.\ degrees admit other characterizations in terms of Martin-L\"of tests
and Ramsey-Type K\"onig's lemmas. For the former, it is well-known
that d.n.c.\ degrees are the degrees of infinite subsets of Martin-L\"of randoms \cite{Kjos-Hanssen2009Infinite,Greenberg2009Lowness}.
The latter has been introduced by Flood in~\cite{Flood2012Reverse} under the name $\rkl$ and and renamed into $\rwkl$
in~\cite{Bienvenu2014Ramsey}. It informally states the existence 
of an infinite subset of $P$ or $\overline{P}$ where $P$ is a path through a tree.

\begin{definition}
Fix a binary tree $T \subseteq 2^{<\Nb}$ and a $c \in \set{0,1}$.
A string $\sigma \in 2^{<\Nb}$ is \emph{homogeneous for a path through $T$ with color $c$} if there exists a $\tau \in T$
such that $\forall i < \card{\sigma}$, $\sigma(i) = 1 \imp \tau(i) = c$. A set $H$ is \emph{homogeneous for a path in
$T$} if there is a $c \in \set{0,1}$ such that for every initial segment $\sigma$ of $H$, $\sigma$
is homogeneous for a path in $T$ with color $c$.
$\rwwkl$ is the statement ``Every tree $T$ of positive measure has an infinite set homogeneous for a path through $T$''.
\end{definition}

Flood proved in~\cite{Flood2012Reverse} that $\rca \vdash \rwwkl \imp \dnr$.
Bienvenu et al.\ proved in~\cite{Bienvenu2014Ramsey} the reverse implication.

\begin{definition}
A \emph{Martin-L\"of test} relative to $X$ is a sequence $(U_i)_{i \in \Nb}$ 
of uniformly $\Sigma^{0,X}_1$ classes such that $\mu(U_n) \leq 2^{-n}$ for all $n$.
A set $H$ is \emph{homogeneous} for a Martin-L\"of test $(U_i)_{i \in \Nb}$
if there exists an $i$ such that $H$ is homogeneous for a path through the tree
corresponding to the closed set~$\overline{U_i}$.
\end{definition}

\begin{theorem}[Flood \cite{Flood2012Reverse}, Bienvenu \& al.~\cite{Bienvenu2014Ramsey}]\label{thm:dnrzp-rwwkl}
For every $n \in \Nb$, the following are equivalent over $\rca+\isig^0_{n+1}$:
\begin{itemize}
  \item[(i)] $\dnr[0^{(n)}]$
  \item[(ii)] Every Martin-L\"of test $(U_i)_{i \in \Nb}$ relative to $\emptyset^{(n)}$
  has an infinite homogeneous set.
  \item[(iii)] Every $\Delta^0_{n+1}$ tree of positive measure has an infinite set homogeneous for a path.
\end{itemize}
\end{theorem}

In the rest of this section, we will prove an equivalent theorem
for $\srrt^2_2$.

\begin{definition}[Downey \& Hirschfeldt~\cite{Downey2010Algorithmic}]
A Martin-L\"of test $(U_n)_{n \in \Nb}$ relative to $X$ 
is a \emph{Schnorr test} relative to $X$ if the measures $\mu(U_n)$
are uniformly $X$-computable.
\end{definition}

\begin{lemma}[$\rca+\bst$]
For every set $A$, every $n \in \Nb$ and every function $f \leq_T A'$
there exists a tree $T \leq_T A'$ such that $\mu(T)$ is an $A'$-computable positive real, $\mu(T) \geq 1 - \frac{1}{2^n}$
and every infinite set homogeneous for a path through $T$ computes a function $g$
such that $g(e) \neq f(e)$ for every~$e$.

Moreover the index for $T$ and for its measure can be found effectively from $n$
and $f$.
\end{lemma}
\begin{proof}
Fix $n \in \Nb$. Let $(D_{e,i})_{e, i \in \Nb}$ be an enumeration of finite sets such that
\begin{itemize}
  \item[(i)] $min(D_{e,i}) \geq i$
  \item[(ii)] $\card{D_{e,i}} = i+2+n$
  \item[(iii)] given an $i$ and finite set $U$ satisfying (i) and (ii),
  one can effectively find an $e$ such that $D_{e, i} = U$.
\end{itemize}

For any canonical index $e$ of a finite set, define $T_e$ to be the downward closure
of the $f$-computable set $\set{\sigma \in \str :  \exists a, b \in D_{f(e), e} : \sigma(a) = 0 \wedge \sigma(b) = 1}$.
The set~$T_e$ exists by~$\bsig^{0,f}_1$, hence~$\bst$.
Define also $T_{\leq e} = \bigcap_{i=0}^e T_e$. It is easy to see that
$$
  \mu(T_e) = 1 - \frac{1}{2^{\card{D_{f(e), e}}-1}}
$$

Fix a $\emptyset'$-computable function $f$. Consider the following tree $T = \bigcap_{i = 0}^{\infty} T_i$.
Because of condition (ii),
$$
\mu(T) \geq 1 - \sum_{i = 0}^{\infty} [1 - \mu(T_i)] 
  = 1 - \sum_{i = 0}^{\infty} \frac{1}{2^{i+1+n}} = 1 - \frac{1}{2^{n}}
$$

\begin{claim}
$T$ is an $f$-computable tree.
\end{claim}
\begin{proof}
Fix a string $\sigma \in \str$.
$\sigma \in T$ iff $\sigma \in \bigcap_{i=0}^{\infty} T_i$
By definition, $\sigma \in T_i$ iff $\sigma \preceq \tau$
for some $\tau \in \str$ such that there are some elements $a, b \in D_{f(i), i}$
verifying $\tau(a) = 0$ and $\tau(b) = 1$.
When $i \geq \card{\sigma}$, because of conditions (i) and (ii)
there exists $a, b \geq i$ with $a, b \in D_{f(i), i}$
and $\tau \succeq \sigma$ such that $\tau(a) = 0$ and $\tau(b) = 1$.
Hence
$\sigma \in T$ iff $\sigma \in T_{\leq \card{\sigma}}$,
which is an $f$-computable predicate uniformly in $\sigma$.
\end{proof}

\begin{claim}
$\mu(T)$ is an $f$-computable real.
\end{claim}
\begin{proof}
Fix any $c \in \Nb$. For any $d \in \Nb$, by condition (ii)
$$
\mu(T_{\leq d}) \geq \mu(T) \geq \mu(T_{\leq d}) - \sum_{i=d}^{\infty} \frac{1}{2^{i+1+n}}
$$
In particular, for $d$ such that $2^{-n} - \sum_{i=0}^{d} \frac{1}{2^{i+1+n}} \leq 2^{-c}$ we have
$$
\card{\mu(T_{\leq d}) - \mu(T)} \leq \sum_{i=d}^{\infty} \frac{1}{2^{i+1+n}} \leq \frac{1}{2^c}
$$
It suffices to notice that $\mu(T_{\leq d})$ is easily $f$-computable
as for $u = max(\bigcup_{i=0}^d D_{f(i), i})$
$$
\mu(T_{\leq d}) = \frac{\card{\set{\sigma \in 2^u : \sigma \in T_{\leq d}}}}{2^u}
$$
\end{proof}

Let $H$ be an infinite set homogeneous for a path through $T$.

\begin{claim}
$H$ computes a function $g$ such that $g(i) \neq f(i)$ for every $i$.
\end{claim}
\begin{proof}
Let $g$ be the $H$-computable function which on input $i$
returns an $e \in \Nb$ such that $D_{e, i}$ is the set
of the first $i+2+n$ elements of $H$. Such an element can be effectively found
by condition (iii).

Assume for the sake of contradiction that $g(i) = f(i)$ for some $i$.
Then by definition of being homogeneous for a path through $T$, there exists a $j \in \set{0,1}$ and a $\sigma \in T$
such that $\sigma(u) = j$ whenever $u \in H$. In particular, $\sigma \in T_i$.
So there exists $a, b \in D_{f(i), i} = D_{g(i), i} \subset H$ such that
$\sigma(a) = 0$ and $\sigma(b) = 1$. Hence there exists an $a \in H$ such that $\sigma(a) \neq j$. 
Contradiction.
\end{proof}

This last claim finishes the proof.
\end{proof}

\begin{corollary}
For every 2-bounded, computable coloring $f : [\Nb]^2 \to \Nb$
there exists a $\emptyset'$-computable tree $T$ of positive $\emptyset'$-computable measure
such that every infinite set homogeneous for a path through $T$
computes an infinite rainbow for $f$.
\end{corollary}

\begin{corollary}\label{cor:schnorr-srrt22}
For every 2-bounded, computable coloring $f : [\Nb]^2 \to \Nb$
there exists a Schnorr test $(U_i)_{i \in \Nb}$ relative to $\emptyset'$ such that
every infinite set homogeneous for $(U_i)_{i \in \Nb}$ computes
an infinite rainbow for $f$.
\end{corollary}

\begin{theorem}[$\rca+\ist$]\label{thm:konig-srrt}
Fix a set $X$. For every $X'$-computable tree $T$ of positive $X'$-computable measure $\mu(T)$
there exists a uniform family $(X_e)_{e \in \Nb}$ of $\Delta^{0, X}_2$ finite sets
whose sizes are uniformly $X'$-computable
and such that every $(X_e)_{e \in \Nb}$-escaping function
computes an infinite set homogeneous for a path through $T$.
\end{theorem}
\begin{proof}
Consider $X$ to be computable for the sake of simplicity.
Relativization is straightforward.
We denote by $(D_e)_{e \in \Nb}$ the canonical enumeration of all finite sets.
Let $T$ be a $\emptyset'$-computable tree of positive $\emptyset'$-computable measure $\mu(T)$.
For each~$s \in \Nb$, let~$T_s$ be the set of strings~$\sigma \in 2^{<\Nb}$ of length~$s$
and let~$\mu_s(T)$ be the first $s$ bits approximation of~$\mu(T)$.
Consider the following set for each finite set $H \subseteq \Nb$ and $k \in \Nb$.
$$
\bad( H, k) = \set{ n \in \Nb : \mu_{4k}(T \cap \Gamma^0_H \cap \Gamma^0_n) < 2^{-2k}}
$$

First notice that the measure of $T \cap \Gamma^0_H$ (resp. $T \cap \Gamma^0_H \cap \Gamma^0_n$)
is $\emptyset'$-computable  uniformly in $H$ (resp. in $H$ and $n$), 
so one $\bad(H, k)$ is uniformly $\Delta^0_2$. We now prove that $\bad(H,k)$
has a uniform $\Delta^0_2$ upper bound, which is sufficient to deduce that~$|\bad(H, k)|$
is uniformly~$\Delta^0_2$.

Given an $H$ and a $k$, let $\epsilon = 2^{-k-1} - 2^{-2k} - 2^{-4k}$.
We can $\emptyset'$-computably find a length $s = s(H, k)$ such that
$$
\frac{|T_s \cap \Gamma^0_H|}{2^s} - \mu(T \cap \Gamma^0_H)  < \epsilon
$$

\begin{claim}
If $2^{-k} \leq \mu(T \cap \Gamma^0_H)$, then $max(\bad(H,k)) \leq s$
\end{claim}
\begin{proof}
Fix any $n > s$. By choice of $s$,
$$
\mu(T \cap \Gamma^0_H \cap \Gamma^1_n) \leq \frac{|T_s \cap \Gamma^0_H|}{2^{s+1}}
	\leq \frac{\mu(T \cap \Gamma^0_H)}{2} + \epsilon
$$
Furthermore,
$$
\mu(T \cap \Gamma^0_H \cap \Gamma^0_n) = \mu(T \cap \Gamma^0_H) - \mu(T \cap \Gamma^0_H \cap \Gamma^1_n)
$$
Putting the two together, we obtain
\begin{eqnarray*}
\mu(T \cap \Gamma^0_H \cap \Gamma^0_n) &\geq& \mu(T \cap \Gamma^0_H) - \frac{\mu(T \cap \Gamma^0_H)}{2} - \epsilon\\
	&\geq& \frac{\mu(T \cap \Gamma^0_H)}{2} - \epsilon \geq 2^{-k-1} - \epsilon \geq 2^{-2k} + 2^{-4k}
\end{eqnarray*}
In particular
$$
\mu_{4k}(T \cap \Gamma^0_H \cap \Gamma^0_n) \geq \mu(T \cap \Gamma^0_H \cap \Gamma^0_n) - 2^{-4k} \geq 2^{-2k}
$$
Therefore $n \not \in \bad(H,k)$.
\end{proof}

For each~$H$ and~$k$, let~$X_{H,k} = \bad(H, k) \cap [0, s(H, k)]$.
The set $X_{H, k}$ is ~$\Delta^0_2$ uniformly in~$H$ and~$k$, and its size is uniformly $\Delta^0_2$.
In addition, by previous claim, if $2^{-k} \leq \mu(T \cap \Gamma^0_H)$ then $\bad(H,k) \subseteq X_{H,k}$.

Let $g : \Pcal_{fin}(\Nb) \times \Nb \times \Nb \to \Nb$ be a total function such that for every finite set $H$
and $k \in \Nb$, $g(H, k, n) \not \in X_{H,k}$ whenever $n \geq \card{X_{H, k}}$. Fix any $k \in \Nb$ 
such that $2^{-k} \leq \mu(T)$.
We construct by $\isig^{0,g}_1$ a set $H$ and a sequence of integers $k_0, k_1, \dots$ by finite approximation as follows.
First let $H_0 = \emptyset$ and~$k_0 = k$. We will ensure during the construction that for all $s$:
\begin{itemize}
  \item[(a)] $\card{H_s} = s$
  \item[(b)] $T \cap \Gamma^0_{H_s}$ has measure at least $2^{-k_s}$
  \item[(c)] $H_s \subseteq H_{s+1}$ and every $n \in H_{s+1} \setminus H_s$ is greater than all elements in $H_s$.
\end{itemize}
Suppose $H_s$ has been defined already. The tree $T \cap \Gamma^0_{H_s}$ has measure at least $2^{-k_s}$
and $\card{\bad(H_s, k_s)}$ has at most $2k_s$ elements. Thus $g(H_s, k_s) \not \in X_{H_s,k_s} \supseteq \bad(H_s, k_s)$.
We set $H_{s+1} = H_s \cup \set{g(e, k_s)}$ and~$k_{s+1}$ be the least integer such that~$2^{-k_{s+1}} \leq 2^{-2k_{s}} - 2^{-4k_s}$.
By definition of~$\bad(H_s, k_s)$, $T \cap \Gamma^0_{H_{s+1}}$ has measure at least
$2^{-2k_s}$ with an approximation of~$2^{-4k_s}$, so has measure at least $2^{-k_{s+1}}$.

Let now $H = \bigcup_s H_s$.

\begin{claim}
$H$ is homogeneous for a path through $T$.
\end{claim}
\begin{proof}
Suppose for the sake of contradiction that $H$ is not homogeneous for a path through $T$.
This means that there are only finitely many $\sigma \in T$ such that $H$ is homogeneous for $\sigma$.
Therefore for some level $l$,
$\set{\sigma \in T_l \; \mid \; \forall i \in H \ \sigma(i) = 0}=\emptyset$.
Since $H \cap \{0,..,l\} = H_l \cap \{0,..,l\}$, we in fact have 
$\set{\sigma \in T_l \; \mid \; \forall i \in H_l \ \sigma(i) = 0}=\emptyset$.
 
In other words, $T \cap \Gamma^0_{H_l} = \emptyset$ which contradicts property (b) in the definition of $H_l$
ensuring that $T \cap \Gamma^0_{H_l}$ has measure at least $2^{-k_l}$.
Thus $H$ is homogeneous for a path through $T$.
\end{proof}
\end{proof}

\begin{theorem}
The following are equivalent over $\rca + \ist$:
\begin{itemize}
  \item[(i)] $\srrt^2_2$
  \item[(ii)] Every Schnorr test $(U_i)_{i \in \Nb}$ relative to $\emptyset'$
  has an infinite homogeneous set.
  \item[(iii)] Every $\Delta^0_2$ tree of $\emptyset'$-computable positive measure 
  has an infinite set homogeneous for a path.
\end{itemize}
\end{theorem}
\begin{proof}
$(i) \imp (iii)$ is Theorem~\ref{thm:konig-srrt}
together with Theorem~\ref{thm:srrt22-characterizations}.
$(iii) \imp (ii)$ is obvious
and $(ii) \imp (i)$ is Corollary~\ref{cor:schnorr-srrt22}. 
\end{proof}

Hirschfeldt et al.\ proved in \cite[Theorem 3.1]{Hirschfeldt2008Limit} that
every $X'$-computable martingale $M$ has a set low over $X$ on which $M$ does not succeed.
Schnorr proved in \cite{Schnorr1971Zufalligkeit} that for every Schnorr test $C$ relative to $X'$
there exists an $X'$-computable martingale $M$ such that a set does not succeeds on $M$ iff it passes the test $C$.
By Corollary~\ref{cor:schnorr-srrt22}, there exists an $\omega$-model of $\srrt^2_2$ containing only low sets.
However we will prove it more directly under the form of a low basis theorem for $\emptyset'$-computable
trees of $\emptyset'$-computable positive measure. This is an adaptation of \cite[Proposition 2.1]{Barmpalias2012Randomness}.

\begin{theorem}[Low basis theorem for $\Delta^0_2$ trees]\label{thm:low-tree-exact-measure}
Fix a set $X$. Every $X'$-computable tree of $X'$-computable positive measure
has an infinite path $P$ low over $X$ (i.e., such that $(X \oplus P)' \leq_T X'$).
\end{theorem}
\begin{proof}
Fix $T$, an $X'$-computable tree of 
$X'$-computable positive measure $\mu(T)$. We will define
an $X'$-computable subtree $U$ of measure $\frac{\mu(T)}{2}$
such that any infinite path through $T$ is GL${}_1$ over~$X$. It then suffices
to take any $\Delta^{0,X}_2$ path through $U$ to obtain the desired path low over~$X$.

Let $f$ be an $X'$-computable function that on input $e$
returns a stage $s$ after which $e$ goes into $A'$ for at most
measure $2^{-e-2}\mu(T)$ of oracles~$A$. Given $e$ and $s = f(e)$, the oracles $A$
such that $e$ goes into $A'$ after stage $s$ form a $\Sigma^{0,X}_1$ class $V_e$
of measure $\mu(V_e) \leq 2^{-e-2}\mu(T)$.
Thus~$\mu(\bigcap_e \overline{V_e}) \geq 1 - \sum_e 2^{-e-2}\mu(T) \geq 1 - \frac{\mu(T)}{2}$.
Therefore~$\mu(T \cap \bigcap_e \overline{V_e}) \geq \frac{\mu(T)}{2}$.
One can easily restrict $T$ to a subtree $U$ such that $[U] \subseteq \bigcap_e \overline{V_e}$
and $\mu(U) = \frac{\mu(T)}{2}$.
For any path $P \in [U]$ and any $e \in \Nb$, $e \in P' \biimp e \in P'_{f(e)}$.
Hence $P$ is GL${}_1$ over~$X$.
\end{proof}

\begin{corollary}
There exists an $\omega$-model of $\srrt^2_2$ containing only low sets.
\end{corollary}

\begin{corollary}
There exists an $\omega$-model of $\srrt^2_2$ which is neither a model of $\semo$ nor of $\sts(2)$.
\end{corollary}
\begin{proof}
If every computable stable tournament had a low infinite subtournament then we could build an $\omega$-model $M$
of $\semo + \sads$ having only low sets, but then $M \models \srt^2_2$ contradicting \cite{Downey20010_2}.
Moreover, by Theorem~\ref{thm:sts-no-low} any $\omega$-model of $\sts(2)$ contains a non-low set.
\end{proof}

In fact we will see later that even $\rrt^2_2$ implies neither $\semo$ nor $\sts(2)$ on $\omega$-models.

\subsection{Relations to other principles}

We now relate the stable rainbow Ramsey theorem for pairs
to other existing principles studied in reverse mathematics.
This provides in particular a factorization of existing
implications proofs. For example, both the rainbow Ramsey theorem for pairs
and the stable Erd\H{o}s-Moser theorem are known to imply the omitting partial types principle ($\opt$) over~$\rca$.
In this section, we show that both principles imply~$\srrt^2_2$,
which itself implies~$\opt$ over~$\rca$.
Hirschfeldt \& Shore in~\cite{Hirschfeldt2007Combinatorial}
introduced~$\opt$ and proved its equivalence with~$\hyp$ over~$\rca$.

\begin{theorem}\label{thm:srrt22-imp-opt}
$\rca \vdash \srrt^2_2 \imp \hyp$
\end{theorem}
\begin{proof}[Proof using Cisma \& Mileti construction, $\rca$]
We prove that the construction from Csima \& Mileti in~\cite{Csima2009strength} that $\rca \vdash \rrt^2_2 \imp \hyp$
produces a rainbow-stable coloring. We take the notations and definitions
of the proof of Theorem~4.1 in~\cite{Csima2009strength}. It is therefore essential
to have read it to understand what follows.
Fix an $x \in \Nb$. By $\bsig^0_1$ there exists an $e \in \Nb$ and a stage $t$ after which
$n^k_j$ and $m^k$ will remains stable for any $k \leq e$ and any $j \in \Nb$
and such that $n^e_i \leq x < n^e_{i+1}$ for some $i$.
\begin{itemize}
  \item If $i > 0$ then 
$x$ will be part of no pair $(m,l)$ for any requirement and $f(x, s) = \tuple{x,s}$
will be fresh for cofinitely many $s$.

  \item If $i = 0$ and $n^e_j$ is defined for each $j$
such that $j+1 \leq \frac{(n^e_0 - m^e)^2 - (n^e_0 - m^e)}{2}$
then as there are finitely many such $j$, after some finite stage
$x$ will not be paired any more and $f(x, s) = \tuple{x,s}$
will be fresh for cofinitely many $s$.

  \item If $i = 0$ and $n^e_j$ is undefined for some $j$
  such that $\tuple{m, x} = j+1$ or $\tuple{x, m} = j+1$ for some $m$, then
  $x$ will be part of a pair $(m, l)$ for cofinitely many $s$
  and so there exists an $m$ such that $f(x,s) = f(m, s)$ for cofinitely many $s$.
  
  \item If $i = 0$ and $n^e_j$ is undefined for some $j$
  such that $\tuple{m, x} \neq j+1$ or $\tuple{x, m} \neq j+1$ for any $m$
  then $x$ will not be paired after some stage
  and $f(x, s) = \tuple{x,s}$ will be fresh for cofinitely many $s$.
\end{itemize}
In any case, either $f(x, s)$ is fresh for cofinitely many $s$,
or there is a $y$ such that $f(x, s) = f(y, s)$ for cofinitely many $s$.
So the coloring is rainbow-stable.
\end{proof}

\smallskip

We can also adapt the proof using $\Pi^0_1$-genericity to $\srrt^2_2$.

\begin{proof}[Proof using $\Pi^0_1$-genericity, $\rca + \ist$]
Take any incomplete $\Delta^0_2$ set $P$ of PA degree.
The author proved in~\cite{Patey2015Degrees} the existence of
a $\Delta^0_2$ function $f$ such that $P$ does not compute any $f$-diagonalizing function.

Fix any functional $\Psi$.
Consider the $\Sigma^0_2$ class
$$
U = \set{X \in \cs : (\exists e) \Psi^X(e) \uparrow \vee \Psi^X(e) = f(e)}
$$

Consider any $\Pi^0_1$-generic $X$ such that $\Psi^X$ is total.
Either there exists a $X \in U$ in which case $\Psi^X(e) = f(e)$ hence $\Psi^X$ is not an $f$-diagonalizing function.
Or there exists a $\Pi^0_1$ class $F$ disjoint from $U$ and containing $X$. Any member of $F$ computes an $f$-diagonalizing
function. In particular $P$ computes an $f$-diagonalizing function. Contradiction.
\end{proof}

\begin{corollary}
$\rca \vdash \srrt^2_2 \imp \opt$
\end{corollary}

The following theorem is not surprising as by a relativization of Theorem~\ref{thm:srrt22-imp-opt}
to $\emptyset'$, there exists an $\emptyset'$-computable rainbow-stable coloring of pairs
such that any infinite rainbow computes a function hyperimmune relative to $\emptyset'$.
Csima et al.~\cite{Csima2004Bounding} and Conidis~\cite{Conidis2008Classifying} proved
that $\amt$ is equivalent over $\omega$-models to the statement ``For any $\Delta^0_2$ function $f$, there exists a function $g$
not dominated by $f$''. Hence any $\omega$-model of $\srrt^2_2[\emptyset']$ is an $\omega$-model of $\amt$.
We will prove that the implication holds over $\rca$.

\begin{theorem}\label{thm:srrt2n-stsn}
$\rca \vdash (\forall n)[\srrt^{n+1}_2 \imp \sts(n)]$
\end{theorem}
\begin{proof}
Fix some~$n \in \Nb$ and let $f : [\Nb]^{n} \to \Nb$ be a stable coloring. 
If~$n = 1$, then~$f$ has a $\Delta^{0,f}_1$ infinite thin set, so suppose~$n > 1$.
We build a $\Delta^{0,f}_1$
rainbow-stable 2-bounded coloring $g : [\Nb]^{n+1} \to \Nb$
such that every infinite rainbow for $g$ is, up to finite changes, thin for $f$.
Construct~$g$ as in the proof of Theorem~\ref{thm:rrt2n-tsn}.
It suffices to check that~$g$ is rainbow-stable whenever $f$ is stable.

Fix some~$x \in \Nb$ and~$\vec z \in [\Nb]^{n-1}$ such that~$x < min(\vec z)$. 
As~$f$ is stable, there exists a stage~$s_0 > max(\vec z)$ after which $f(\vec z, s) = f(\vec z, s_0)$.
Interpret~$f(\vec z, s_0)$ as a tuple~$\tuple{u,v}$.
If~$u \geq v$ or~$v \geq min(\vec z)$ or~$x \not \in \{u,v\}$, then $g(x, \vec z, s)$ will be given
a fresh color for every~$s \geq s_0$.
If~$u < v < min(\vec z)$ and~$x \in \{u,v\}$ (say $x = u$),
then~$g(x, \vec z, s) = g(v, \vec z, s)$ for every~$s \geq v$.
Therefore $g$ is rainbow-stable.
\end{proof}

\begin{corollary}
$\rca \vdash \srrt^3_2 \imp \amt$
\end{corollary}

\begin{remark}\label{rem:dnrzk-rrtk}\ 
As Bienvenu et al.\ \cite{Bienvenu2014role} proved that there is a computable instance of $\amt$ such that
no 2-random bounds a solution to it, we obtain as a corollary that the reverse implication
of Corollary~\ref{cor:rrtk-dnrzk} does not hold.
\end{remark}

\begin{theorem}[$\rca + \bst$]
For every $\Delta^0_2$ function $f$, 
there exists a computable stable coloring $c:[\Nb]^2 \to \Nb$ such that every infinite
set thin for $c$ computes an $f$-diagonalizing function.
\end{theorem}
\begin{proof}
Fix a $\Delta^0_2$ function $f$ as stated above.
For any $n \in \Nb$, fix a canonical enumeration $(D_{n, e})_{e \in \Nb}$ of all
finite sets of $n+1$ integers greater than $n$.
We will build a computable stable coloring $c : [\Nb]^2 \to \Nb$ fulfilling 
the following requirements for each $e, i \in \Nb$:

\smallskip
$\Rcal_{e, i}: $ $\exists u \in D_{\tuple{e,i}, f(e)}$ such that $(\forall^{\infty} s)c(u, s) = i$.
\smallskip

We first check that if every requirement is satisfied, then any
infinite set thin for $c$ computes an $f$-diagonalizing function.
Let $H$ be an infinite set thin for $c$ with witness color $i$.
Define $h: \Nb \to \Nb$ to be the $H$-computable function which on $e$
returns the value $v$ such that $D_{\tuple{e,i}, v}$
is the set of the $\tuple{e,i}+1$ first elements of $H$ greater than
$\tuple{e,i}$.

\begin{claim}
$h$ is an $f$-diagonalizing function.
\end{claim}
\begin{proof}
Suppose for the sake of absurd that $h(e) = f(e)$ for some $e$.
Then $D_{\tuple{e,i}, h(e)} = D_{\tuple{e,i}, f(e)}$.
But by $\Rcal_{e,i}$, $\exists u \in D_{\tuple{e,i}, f(e)}$ such that
$(\forall^{\infty} s)c(u,s) = i$. Then there is an $s \in H$ such that
$c(u, s) = i$, and as $D_{\tuple{e,i}, f(e)} \cup \set{s} \subset H$,
$H$ is not thin for $c$ with witness $i$. Contradiction.
\end{proof}

By Schoenfield's limit lemma, let $g(\cdot, \cdot)$ be the partial approximations of~$f$.
The strategy for satisfying a local requirement $\Rcal_{e,i}$ is as follows.
At stage $s$, it takes the least element $u$ of $D_{\tuple{e,i}, g(x,s)}$ not restrained by a strategy of
higher priority if it exists. Then it puts a restraint on $u$ and \emph{commits}
$u$ to assigning color $i$. For any such $u$, this commitment 
will remain active as long as the strategy has a restraint on that element.
Having done all this, the local strategy is declared to be satisfied and will not act again, unless either a higher priority
puts a restraint on $u$, or releases a $v \in D_{\tuple{e,i}, g(e,s)}$ with $v < u$ or 
at a further stage $t > s$, $g(e,t) \neq g(e, s)$.
In each case, the strategy gets \emph{injured} and has to reset, releasing its restraint.

To finish stage $s$, the global strategy assigns $c(u, s)$ for all $u \leq s$ as follows:
if $u$ is commited to some assignment of $c(u, s)$ due to a local strategy, define $c(u, s)$ to be this value.
If not, let $c(u, t) = 0$. This finishes the construction and we now turn to the verification.
It is easy to check that each requirement restrains at most one element at a given stage.

\begin{claim}
Each strategy $\Rcal_{e,i}$ acts finitely often.
\end{claim}
\begin{proof}
Fix some strategy~$\Rcal_{e,i}$.
By $\bst$, there is a stage~$s_0$ after which $g(x,s) = f(x)$ for every~$x \leq \tuple{e,i}$.
Each strategy restrains at most one element, 
and the strategies of higher priority will always choose the same elements after stage~$s_0$.
As $\card{D_{\tuple{e,i}, f(e)}} = \tuple{e,i}+1$,
the set of $u \in D_{\tuple{e,i}, f(e)}$ such that no strategy of higher priority puts a restraint on $u$ is non empty
and does not change.
Let $u_{min}$ be its minimal element. By construction, $\Rcal_{e,i}$ will choose $u_{min}$ before stage $s_0$
and will not be injured again.
\end{proof}

\begin{claim}
The resulting coloring $c$ is stable.
\end{claim}
\begin{proof}
Fix a $u \in \Nb$.
If $\tuple{e, i} > u$ then $\Rcal_{e,i}$ does not put a restraint on $u$ at any stage.
As each strategy acts finitely often, by $\bst$ there exists a stage $s_0$ after which
no strategy $\Rcal_{e, i}$ with $\tuple{e,i} \leq u$ will act on $u$.
There are two cases: In the first case, at stage $s_0$ the element $u$ is restrained
by some strategy $\Rcal_{e,i}$ with $\tuple{e,i} \leq u$ in which case 
$c(u, s)$ will be assigned a unique color specified by strategy $\Rcal_{e,i}$
for cofinitely many $s$. In the other case, after stage $s_0$, the element $u$ is free from
any restraint, and $c(u, s) = 0$ for cofinitely many $s$.
\end{proof}
\end{proof}

\begin{corollary}
$\rca + \ist \vdash \sts(2) \imp \srrt^2_2$
\end{corollary}

\begin{theorem}[$\rca$]
For every rainbow-stable 2-bounded coloring $f : [\Nb]^2 \to \Nb$,
there exists an $f$-computable stable tournament $T$ such that every 
infinite transitive subtournament of $T$ computes a rainbow for~$f$.
\end{theorem}
\begin{proof}
Use exactly the same construction as in Theorem~3.1 in~\cite{Kang2014Combinatorial}.
We will prove that in case of rainbow-stable colorings, the constructed tournament $T$ is stable.
Fix an $x \in \Nb$. By rainbow-stability, either $f(x, s)$ is a fresh color for cofinitely many $s$,
in which case $T(x, s)$ holds for cofinitely many $s$, or there exists a $y$ such that
$f(y, s) = f(x, s)$ for cofinitely many $s$. If $T(x, y)$ holds then
$T(x, s)$ does not hold and $T(y, s)$ holds for cofinitely many $s$.
Otherwise $T(x, s)$ holds and $T(y, s)$ does not hold for cofinitely many $s$.
Hence $T$ is stable.\end{proof}

\begin{corollary}
$\rca \vdash \semo \imp \srrt^2_2$
\end{corollary}

\begin{question}
Does $\srrt^2_2 + \coh$ imply $\rrt^2_2$ over $\rca$ ?
\end{question}

\section{Weakly stable rainbow Ramsey theorem}

Despite the robustness of the stable rainbow Ramsey theorem for pairs
which has been shown to admit several simple characterizations,
rainbow-stability does not seem to be the natural stability notion corresponding
to~$\rrt^2_2$. In particular, it is unknown
whether~$\rca \vdash \coh + \srrt^2_2 \imp \rrt^2_2$.
In this section, we study another more general notion of stability introduced by Wang in~\cite{Wang2014Cohesive}
and which, together with cohesiveness, recovers the full raibow Ramsey's theorem for pairs.
However, this notion of stability does not admit as simple characterizations
as for~$\srrt^2_2$.

Recall that a coloring $f : [\Nb]^2 \to \Nb$ is \emph{weakly rainbow-stable} if
$$
(\forall x)(\forall y)[(\forall^{\infty} s)f(x, s) = f(y,s) \vee (\forall^{\infty} s)f(x, s) \neq f(y,s)]
$$
It is easy to see that every rainbow-stable coloring is weakly rainbow-stable,
hence $\rca \vdash \wsrrt^2_2 \imp \srrt^2_2$.
Wang proved in \cite[Lemma 4.11]{Wang2014Cohesive} that $\rca \vdash \coh + \wsrrt^n_2 \imp \rrt^n_2$
and that $\wsrrt^2_2[\emptyset']$ has an $\omega$-model
with only low${}_2$ sets.

We have seen in Lemma~\ref{lem:rrtn-jump} that~$\rca \vdash (\forall n)(\rrt^{n+1}_2 \imp \rrt^n_2[\emptyset'])$.
We show through the following theorem that~$\wsrrt^{n+1}_2$ corresponds to the exact strength
of~$\rrt^n_2[\emptyset']$ for every~$n$.

\begin{theorem}
For every standard $n \geq 1$, $\rca + \bst \vdash \wsrrt^{n+1}_2 \biimp \rrt^n_2[\emptyset']$.
\end{theorem}
\begin{proof}
For the direction, $\rca \vdash \wsrrt^{n+1}_2 \imp \rrt^n_2[\emptyset']$, 
simply notice that the coloring of $(n+1)$-tuples constructed
in Lemma~\ref{lem:rrtn-jump} is weakly rainbow-stable. We will prove the converse over~$\rca +\bst$.
Let $f : [\Nb]^{n+1} \to \Nb$ be a 2-bounded weakly rainbow-stable coloring.

Let $g : [\Nb]^n \to \Nb$ be the 2-bounded coloring which on $\vec{x} \in [\Nb]^n$
will fetch the least $\vec{y} \preceq_n \vec{x}$ such that $(\forall^{\infty} s)f(\vec{x}, s) = f(\vec{y}, s)$
and return color $\tuple{\vec{y}}$. One easily sees that $g$ is $f'$-computable and 2-bounded.
By $\rrt^n_2[\emptyset']$, let $H$ be an infinite rainbow for $g$. 
We claim that $H$ is a prerainbow for~$f$.
Suppose for the sake of contradiction that there exists $\vec{x} \preceq_n \vec{y} \in H$
such that $(\forall^{\infty} s)f(\vec{x}, s) = f(\vec{y}, s)$. Then by definition
$g(\vec{x}) = g(\vec{y}) = \tuple{\vec{x}}$ and $H$ is not a rainbow for $g$.
By Lemma~\ref{lem:prerainbow-equiv} and~$\bst$, $f \oplus H$ computes an infinite $f$-rainbow.
\end{proof}

\begin{lemma}[$\rca+\ist$]\label{lem:char1-wsrrt22}
For every computable weakly rainbow-stable 2-bounded coloring $f : [\Nb]^2 \to \Nb$
there exists a uniform family $(X_e)_{e \in \Nb}$ of $\Delta^0_2$ finite sets
such that every $(X_e)_{e \in \Nb}$-escaping function computes an infinite $f$-rainbow.
\end{lemma}
\begin{proof}
Fix any uniform family $(D_e)_{e \in \Nb}$ of finite sets.
Let $f : [\Nb]^2 \to \Nb$ be a 2-bounded weakly rainbow-stable computable coloring.
For an element $x$, define 
$$
\bad(x) = \set{ y \in \Nb : (\forall^\infty s)c(x,s) = c(y,s)}
$$
Notice that $x \in \bad(x)$.
Because $f$ is weakly rainbow-stable, $\bad$ is a $\Delta^0_2$ function.
For a set $S$, $\bad(S) = \bigcup_{x \in S} \bad(x)$.
Define $X_e = \bad(D_e)$. Hence $X_e$ is a $\Delta^0_2$ set, and this uniformly in $e$.
Moreover, $\card{X_e} \leq 2\card{D_e}$.

Let $h : \Nb \to \Nb$ be a function satisfying $(\forall e)(\forall n)[\card{X_e} \leq n \imp h(e, n) \not \in X_e]$.
We can define $g : \Nb \to \Nb$ by $g(e) = h(e, 2{\card{D_e} \choose 2})$.
Hence $(\forall e)g(e) \not \in X_e$.

We construct a prerainbow $R$ by stages $R_0 (=\emptyset) \subsetneq R_1 \subsetneq R_2, \dots$
as in Lemma~\ref{lem:char1-srrt22}.
Assume that at stage $s$, 
$(\forall \{x,y\} \subseteq R_s)(\forall^{\infty} s)[f(x, s)\neq f(y,s)]$.
Because $R_s$ is finite, we can computably find some index $e$ such that $R_s = D_e$.
Set $R_{s+1} = R_s \cup \set{g(e)}$. By definition, $g(e) \not \in X_e$.
Let $x \in R_s$. Because $g(e) \not \in X_e$, $(\forall^{\infty} s) f(x, s) \neq f(g(e), s)$.
By~$\ist$, the set~$R$ is a prerainbow for~$f$.
By Lemma~\ref{lem:prerainbow-equiv} we can compute an infinite rainbow for $f$ 
from $R \oplus f$.
\end{proof}

\subsection{Lowness and bushy tree forcing}

In this section, we prove that the rainbow Ramsey theorem for pairs
restricted to weakly rainbow-stable colorings is strictly weaker
than the full rainbow Ramsey theorem for pairs, by constructing
an~$\omega$-model of~$\wsrrt^2_2$ having only low set.
As~$\rrt^2_2$ does not admit such a model,
$\wsrrt^2_2$ does not imply $\rrt^2_2$ over~$\rca$.

\begin{theorem}\label{thm:wsrrt22-low}
For every set~$X$ and every weakly rainbow-stable $X$-computable 2-bounded function~$f : [\Nb]^2 \to \Nb$,
there exists an infinite $f$-rainbow low over~$X$.
\end{theorem}

We will use \emph{bushy tree forcing} for building a low solution to a computable
instance of~$\wsrrt^2_2$. This forcing notion has been successfuly used for proving
many properties over d.n.c.\ degrees~\cite{Ambos-Spies2004Comparing,BienvenuEvery,Khan2014Forcing,Patey2015Ramsey}. 
Indeed, the power of a d.n.c.\ function is known
to be equivalent to finding a function escaping a uniform family of c.e.\ sets~\cite{Kjos-Hanssen2011Kolmogorov},
which is exactly what happens with bushy tree forcing: we build an infinite set by finite approximations,
avoiding a set of bad extensions whose size is computably bounded.
We start by stating the definitions of bushy tree forcing and the basic properties
without proving them. See the survey of Kahn \& Miller~\cite{Khan2014Forcing} for a good introduction.

\begin{definition}[Bushy tree]
Fix a function $h$ and a string $\sigma \in \Nb^{<\Nb}$.
A tree $T$ is  $h$-bushy above $\sigma$ if every $\tau \in T$ is increasing
and comparable with $\sigma$ and whenever $\tau \succeq \sigma$ is not a leaf of~$T$,
it has at least $h(\card{\tau})$ immediate children. 
We call $\sigma$ the \emph{stem} of $T$.
\end{definition}

\begin{definition}[Big set, small set]
Fix a function $h$ and some string $\sigma \in \Nb^{<\Nb}$.
A set $B \subseteq \Nb^{<\Nb}$ is \emph{$h$-big above $\Nb$} 
if there exists a finite tree $T$ $h$-bushy above $\sigma$ 
such that all leafs of $T$ are in $B$. 
If no such tree exists, $B$ is said to be \emph{$h$-small above $\sigma$}.
\end{definition}

Consider for example a weakly rainbow-stable 2-bounded function~$f : [\Nb]^2 \to \Nb$.
We want to construct an infinite prerainbow for~$f$. We claim that
the following set is~$id$-small above~$\epsilon$, where~$id$ is the identity function:
\[
B_f = \{ \sigma \in \Nb^{<\Nb} : (\exists x, y \in \sigma)(\forall^{\infty} s)f(x,s) = f(y, s)\}
\]
Indeed, given some string~$\sigma \not \in B_f$,
there exists at most~$|\sigma|$ integers~$x$ such that~$\sigma x \in B_f$.
Therefore, given any infinite tree which is $h$-bushy above~$\emptyset$,
at least one of the paths will be a prerainbow for~$f$. Also note that because
$f$ is weakly rainbow-stable, the set~$B_f$ is~$\Delta^{0,f}_2$.
We now state some basic properties about bushy tree forcing.


\begin{lemma}[Smallness additivity] \label{lem:smallness-add}
Suppose that $B_1, B_2, \ldots, B_n$ are subsets of $\Nb^{<\Nb}$, 
$g_1$, $g_2$, ..., $g_n$ are functions, and $\sigma \in \Nb^{<\Nb}$.
If $B_i$ is $g_i$-small above~$\sigma$ for all~$i$, then $\bigcup_i B_i$ is $(\sum_i g_i)$-small above $\sigma$. 
\end{lemma}

\begin{lemma}[Small set closure]\label{lem:small-set-closure}
We say that $B  \subseteq \Nb^{<\Nb}$ is \emph{$g$-closed} if whenever $B$ is $g$-big above a string $\rho$ then $\rho \in B$. Accordingly, the \emph{$g$-closure} of any set~$B \subseteq \Nb^{<\Nb}$ is the set $C = \set{\tau \in \Nb^{<\Nb} : B \mbox{ is $g$-big above } \tau}$. If $B$ is $g$-small above a string~$\sigma$, then its closure is also $g$-small above $\sigma$.
\end{lemma}

Note that if~$B$ is a $\Delta^{0,X}_2$ $g$-small set for some computable function~$g$,
so is the $g$-closure of~$B$. Moreover, one can effectively find a $\Delta^{0,X}_2$ index of the $g$-closure of~$B$
given a $\Delta^{0,X}_2$ index of~$B$.
Fix some set~$X$.
Our forcing conditions are tuples~$(\sigma, g, B)$ where
$\sigma$ is an increasing string, $g$ is a computable function and~$B \subseteq \Nb^{<\Nb}$ is a $\Delta^{0,X}_2$
$g$-closed set $g$-small above~$\sigma$.
A condition~$(\tau, h, C)$ \emph{extends} $(\sigma, g, B)$ if~$\sigma \preceq \tau$ and~$B \subseteq C$.
Any infinite decreasing sequence of conditions starting with~$(\epsilon, id, B_f)$
will produce a prerainbow for~$f$.

The following lemma is sufficient to deduce the existence of a $\Delta^{0,X}_2$ infinite prerainbow for~$f$.

\begin{lemma}\label{lem:wsrrt22-ext}
Given a condition~$(\sigma, g, B)$, one can $X'$-effectively
find some~$x \in \Nb$ such that the condition~$(\sigma x, g, B)$ is a valid extension.
\end{lemma}
\begin{proof}
Pick the first~$x \in \Nb$ greater than $\sigma(|\sigma|)$ such that~$\sigma x \not \in B$.
Such~$x$ exists as there are at most~$g(|\sigma|)-1$ many bad~$x$ by~$g$-smallness of~$B$.
Moreover~$x$ can be found $X'$-effectively as~$B$ is~$\Delta^{0,X}_2$.
By $g$-closure of~$B$, $B$ is $g$-small above~$\sigma x$. Therefore~$(\sigma x, g, B)$
is a valid extension.
\end{proof}

A sequence~$G$ \emph{satisfies} the condition~$(\sigma, g, B)$
if it is increasing, $\sigma \prec G$ and~$B$ is~$g$-small above~$\tau$
for every~$\tau \prec G$.
We say that~$(\sigma, g, B) \Vdash \Phi^{G \oplus X}_e(e) \downarrow$
if~$\Phi^{\sigma \oplus X}_e(e) \downarrow$, and~$(\sigma, g, B) \Vdash \Phi^{G \oplus X}_e(e) \uparrow$
if~$\Phi^{G \oplus X}_e(e) \uparrow$ for every sequence~$G$ satisfying the condition~$(\sigma, g, B)$.
The following lemma decides the jump of the infinite set constructed.

\begin{lemma}\label{lem:wsrrt22-force}
Given a condition~$(\sigma, g, B)$ and an index~$e \in \Nb$,
one can $X'$-effectively find some extension~$d = (\tau, h, C)$ 
such that~$d \Vdash \Phi^{G \oplus X}_e(e) \downarrow$ or
$d \Vdash \Phi^{G \oplus X}_e(e) \uparrow$. 
Moreover, one can~$X'$-decide which of the two holds.
\end{lemma}
\begin{proof}
Consider the following~$\Sigma^{0,X}_1$ set:
\[
D = \{ \tau \in \Nb^{<\Nb} : \Phi^{\tau \oplus X}_e(e) \downarrow \}
\]
The question whether~$D$ is~$g$-big above~$\sigma$ is~$\Sigma^{0,X}_1$ and therefore can be~$X'$-decided.
\begin{itemize}
	\item If the answer is yes, we can~$X$-effectively find a finite tree~$T$ $g$-bushy above~$\sigma$
	witnessing this. As~$B$ is~$\Delta^{0,X}_2$, we can take~$X'$-effectively some leaf~$\tau \in T$.
	By definition of~$T$, $\sigma \prec \tau$. As~$B$ is $g$-closed, $B$ is $g$-small above~$\tau$,
	and therefore~$(\tau, g, B)$ is a valid extension. Moreover~$\Phi^{\tau \oplus X}_e(e) \downarrow$.
	\item If the answer is no, the set~$D$ is $g$-small above~$\sigma$.
	By the smallness additivity property (Lemma~\ref{lem:smallness-add}),
	$B \cup D$ is $2g$-small above~$\sigma$. We can $X$-effectively find a $\Delta^0_2$
	index for its $2g$-closure~$C$. The condition~$(\sigma, 2g, C)$ is a valid extension
	forcing~$\Phi^{G \oplus X}_e(e) \uparrow$.
\end{itemize}
\end{proof}

We are now ready to prove Theorem~\ref{thm:wsrrt22-low}.

\begin{proof}[Proof of Theorem~\ref{thm:wsrrt22-low}]
Fix some set~$X$ and some weakly rainbow-stable $X$-computable 2-bounded function~$f : [\Nb]^2 \to \Nb$.
Thanks to Lemma~\ref{lem:wsrrt22-ext} and Lemma~\ref{lem:wsrrt22-force}, define an infinite decreasing
$X'$-computable sequence of conditions~$c_0 \geq c_1 \geq \dots$ starting with~$c_0 = (\epsilon, id, B_f)$ and
such that for each~$s \in \Nb$,
\begin{itemize}
	\item[(i)] $|\sigma_s| \geq s$
	\item[(ii)] $c_{s+1} \Vdash \Phi^{G \oplus X}_s(s) \downarrow$ or~$c_{s+1} \Vdash \Phi^{G \oplus X}_s(s) \uparrow$
\end{itemize}
where~$c_s = (\sigma_s, g_s, B_s)$. The set~$G = \bigcup_s \sigma_s$
is a prerainbow for~$f$. By (i), $G$ is infinite and by~(ii), $G$ is low over~$X$. 
By Lemma~\ref{lem:prerainbow-equiv}, $G \oplus X$ computes an infinite $f$-rainbow.
\end{proof}

\begin{corollary}
There exists an $\omega$-model of $\wsrrt^2_2$ having only low sets.
\end{corollary}

\begin{corollary}
$\wsrrt^2_2$ does not imply~$\rrt^2_2$ over~$\rca$.
\end{corollary}
\begin{proof}
By Theorem~\ref{thm-rrt22-dnrzp}, every model of~$\rrt^2_2$ is a model of~$\dnr[\emptyset']$,
and no function d.n.c.\ relative to~$\emptyset'$ is low.
\end{proof}

\subsection{Relations to other principles}

In this last section, we prove that the rainbow Ramsey theorem
for pairs for weakly rainbow-stable colorings
is a consequence of the stable free set theorem for pairs.
We need first to introduce some useful terminology.

\begin{definition}[Wang in~\cite{Wang2014Cohesive}]
Fix a 2-bounded coloring $f : [\Nb]^n \to \Nb$ and~$k \leq n$.
A set~$H$ is a \emph{$k$-tail $f$-rainbow} if $f(\vec{u}, \vec{v}) \neq f(\vec{w}, \vec{x})$
for all~$\vec{u}, \vec{w} \in [H]^{n-k}$ and~ \emph{distinct} $\vec{v}, \vec{x} \in [H]^k$.
\end{definition}

Wang proved in~\cite{Wang2014Cohesive} that for every 2-bounded coloring $f : [\Nb]^n \to \Nb$,
every $f$-random computes an infinite 1-tail $f$-rainbow $H$.
We refine this result by the following lemma.

\begin{lemma}[$\rca$]\label{lem:dnc-normal-rrt}
Let $f : [\Nb]^{n+1} \to \Nb$ be a 2-bounded coloring.
Every function d.n.c.\ relative to $f$ computes an infinite 1-tail $f$-rainbow $H$.
\end{lemma}
\begin{proof}
By~\cite{Kjos-Hanssen2011Kolmogorov}, every function d.n.c.\ relative to $f$ computes a function $g$
such that if $|W^f_e| \leq n$ then $g(e, n) \not \in W^f_e$.
Given a finite 1-tail $f$-rainbow $F$,
there exists at most ${\card{F} \choose n}$ elements $x$ such that $F \cup \{x\}$
is not a 1-tail $f$-rainbow. We can define an infinite 1-tail $f$-rainbow $H$ by stages,
starting with $H_0 = \emptyset$. Given a finite 1-tail $f$-rainbow $H_s$ of cardinal $s$,
set $H_{s+1} = H_s \cup \{g(e, {s \choose n})\}$ where $e$ is a Turing index
such that $W^f_e = \{x : H_s \cup \{x\} \mbox{ is not a 1-tail } f\mbox{-rainbow}\}$.
\end{proof}

\begin{theorem}
$\rca + \bst \vdash \sfs(2) \imp \wsrrt^2_2$
\end{theorem}
\begin{proof}
Fix a weakly rainbow-stable 2-bounded coloring $f : [\Nb]^2 \to \Nb$.
As $\rca \vdash \sfs(2) \imp \dnr$, there exists by Lemma~\ref{lem:dnc-normal-rrt}
an infinite 1-tail $f$-rainbow $X$.
We will construct an infinite $X \oplus f$-computable stable coloring 
$g : [X]^2 \to \set{0,1}$ such that every infinite $g$-free set is an $f$-rainbow.
We define the coloring $g : [\Nb]^2 \to \Nb$ by stages as follows.

At stage $s$, assume $g(x, y)$ is defined for every $x, y < s$.
For every pair $x < y < s$ such that $g(x, s) = g(y, s)$, set $g(y, s) = x$.
For the remaining $x < s$, set $g(x, s) = 0$.
This finishes the construction. We now turn to the verification.

\begin{claim}
Every infinite set $H$ free for $g$ is a rainbow for $f$.
\end{claim}
\begin{proof}
Assume for the sake of contradiction that $H$ is not a rainbow for $f$.
Because $X$ is a 1-tail $f$-rainbow and $H \subseteq X$, 
there exists $x, y, s \in H$ such that $c(x, s) = c(y, s)$
with $x < y < s$. As $f$ is 2-bounded, neither $x$ nor $y$ can be part of another pair
$u, v$ such that $f(u, s) = f(v, s)$. So neither $x$ nor $y$ is restrained by another pair
already satisfied, and during the construction we set $g(y, s) = x$.
So $g(y, s) = x$ with $\set{x, y, s} \subset H$, contradicting freeness of $H$ for $g$.
\end{proof}

\begin{claim}
The coloring $g$ is stable.
\end{claim}
\begin{proof}
Fix a $y \in \Nb$. As $f$ is weakly rainbow-stable, we have two cases.
Either there exists an $x < y$ such that $f(y, s) =  f(x, s)$ for cofinitely
many $s$, in which case $g(y, s) = x$ for cofinitely many $s$ and we are done.
Or $f(y, s) \neq f(x, s)$ for each $x < y$ and cofinitely many $s$.
Then by $\bst$, for cofinitely many $s$, $f(y, s) = 0$.
\end{proof}
\end{proof}

\begin{question}
Does $\sts(2)$ imply $\wsrrt^2_2$ over $\rca$ ?
\end{question}

\vspace{0.5cm}

\noindent \textbf{Acknowledgements}. 
The author is thankful to his PhD advisor Laurent Bienvenu 
and to Wei Wang for useful comments and discussions.

\bibliographystyle{plain}

\appendix

\clearpage

\section{Diagram of relations}

\begin{figure}[htbp]
\caption{Diagram of considered principles modulo $\rca$}
\begin{center}
\includegraphics[width=13cm]{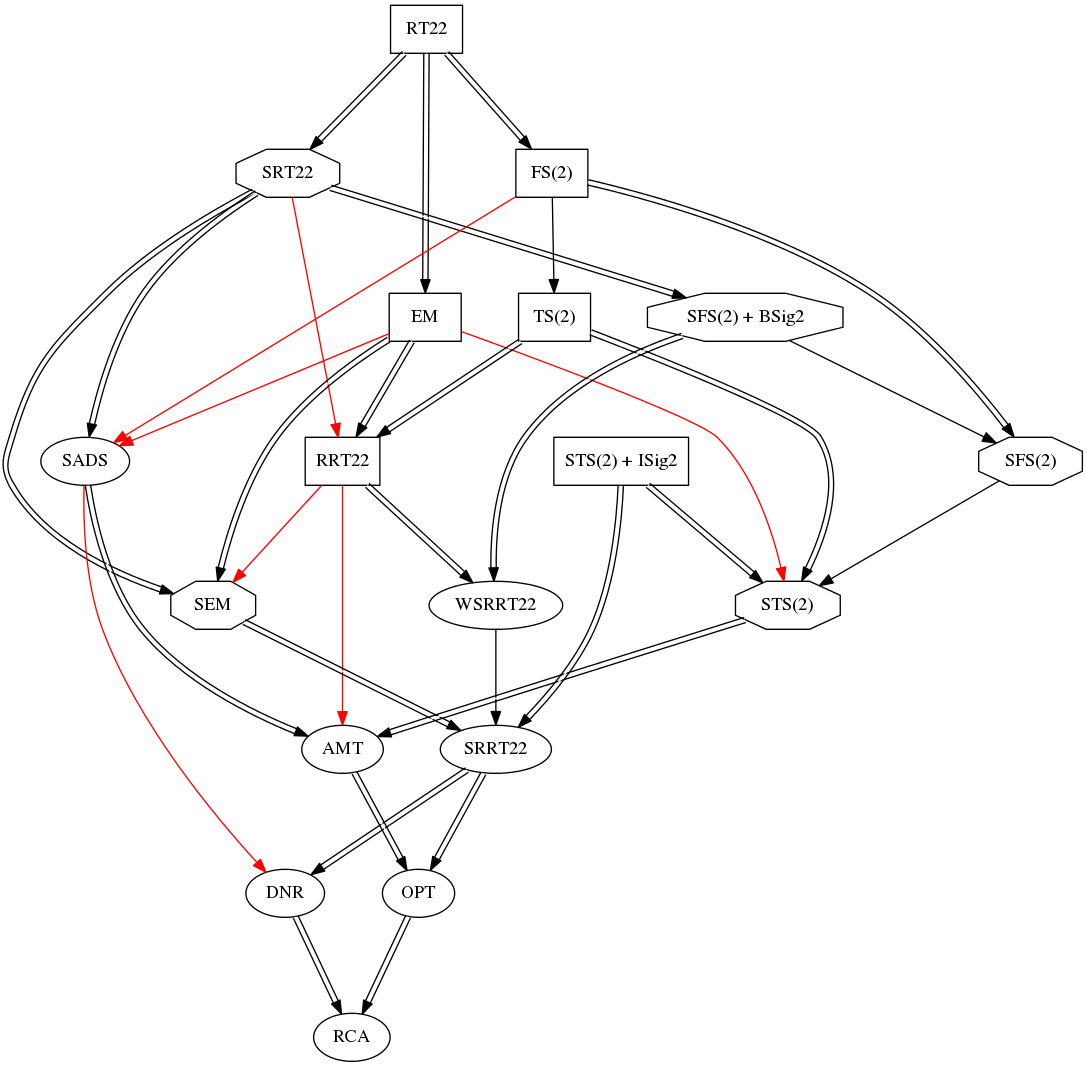}
\end{center}

\smallskip

\begin{itemize}
\begin{minipage}[t]{0.4\linewidth}  
  \item[] $\imp$ Simple implications
  \item[] $\Imp$ Strict implications
  \item[] $\textcolor{red}{\imp}$ Non-implications
\end{minipage}
\begin{minipage}[t]{0.6\linewidth}  
  \item[] Square : Standard model with only low${}_2$ sets
  \item[] Polygon : Model with only low sets
  \item[] Ellipse: Standard model with only low sets
\end{minipage}
\end{itemize}
\end{figure}

\begin{figure}[htbp]
\caption{Diagram of considered principles over $\omega$-models}
\begin{center}
\includegraphics[width=11cm]{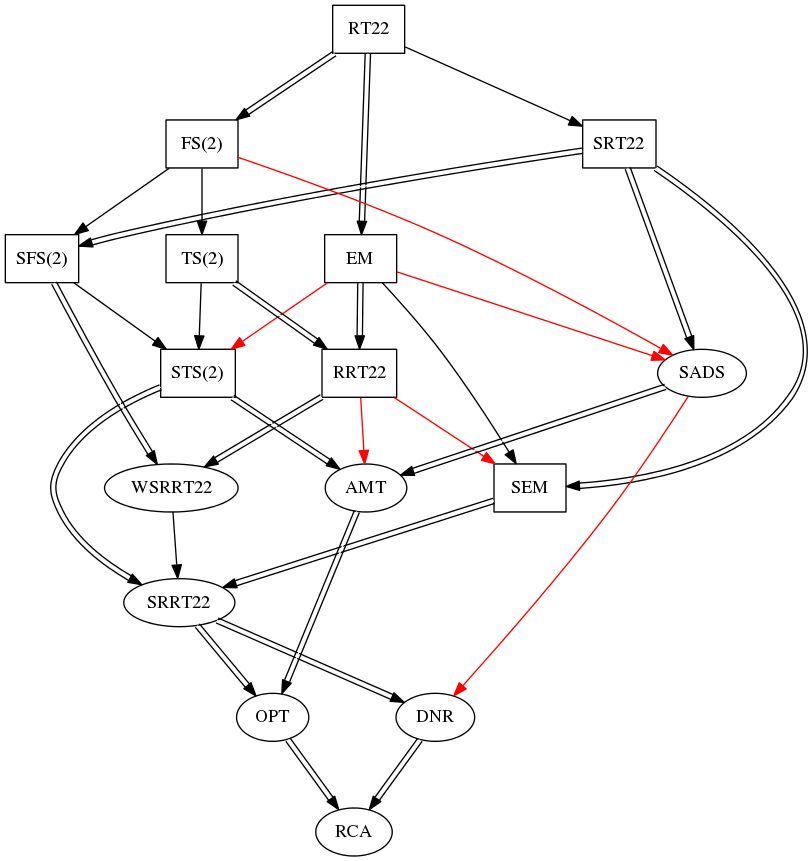}
\end{center}

\smallskip

\begin{itemize}
\begin{minipage}[t]{0.4\linewidth}  
  \item[] $\imp$ Simple implications
  \item[] $\Imp$ Strict implications
  \item[] $\textcolor{red}{\imp}$ Non-implications
\end{minipage}
\begin{minipage}[t]{0.6\linewidth}  
  \item[] Square : Standard model with only low${}_2$ sets
  \item[] Ellipse: Standard model with only low sets
\end{minipage}
\end{itemize}
\end{figure}

\end{document}